\documentclass[final]{amsart}

\usepackage{pxfonts}
\usepackage[T1]{fontenc}
\usepackage[utf8]{inputenc}

\usepackage{xypic}
\usepackage{graphics}
\usepackage{tikz}
\usetikzlibrary{decorations.pathreplacing,backgrounds,arrows,shapes,decorations.pathmorphing,arrows.meta}

\usepackage[bookmarks=true,
colorlinks=true, citecolor=blue,
linkcolor=blue, linktocpage=true]{hyperref}
\usepackage{amssymb,amsthm,mathtools,mathrsfs}
\usepackage{stmaryrd}
\usepackage[color,notref,notcite]{showkeys}

\usepackage{aliascnt}
\newaliascnt{thm}{equation}
\usepackage[capitalize]{cleveref}
\newtheorem{thm}[equation]{Theorem}
\newtheorem{thm-intro}[thm]{Theorem}
\newtheorem*{thm*}{Theorem}
\newtheorem{pro}[thm]{Proposition}
\newtheorem{lem}[thm]{Lemma}
\newtheorem{cor}[thm]{Corollary}
\newtheorem*{cor*}{Corollary}
\theoremstyle{definition}
\newtheorem{dfn}[thm]{Definition}

\newtheorem{rmk}[thm]{Remark}
\numberwithin{equation}{section}

\usepackage[numbers]{natbib}
\setcitestyle{notesep={: }}

\DeclareMathOperator{\coimg}{coim}
\DeclareMathOperator{\coker}{coker}
\DeclareMathOperator{\colim}{colim}
\DeclareMathOperator{\cone}{cone}
\DeclareMathOperator{\dL}{L\!}

\DeclareMathOperator{\dtm}{DTM}
\DeclareMathOperator{\fmd}{proj_{\scriptscriptstyle\mathrm{fil}}}
\DeclareMathOperator{\FMd}{Mod_{\scriptscriptstyle\mathrm{fil}}}
\DeclareMathOperator{\FP}{Proj_{\scriptscriptstyle\mathrm{fil}}}
\DeclareMathOperator{\fper}{{\mathcal D}^{\scriptscriptstyle\mathrm{perf}}_{\scriptscriptstyle\mathrm{fil}}}
\newcommand{\fperalt}{\fper(R)}

\DeclareMathOperator{\GMd}{Mod_{\scriptscriptstyle\mathrm{gr}}}
\DeclareMathOperator{\ihom}{\underline{hom}}
\DeclareMathOperator{\img}{im}
\DeclareMathOperator{\Md}{Mod}
\DeclareMathOperator{\md}{proj}
\DeclareMathOperator{\per}{{\mathcal D}^{\scriptscriptstyle\mathrm{perf}}}
\DeclareMathOperator{\seq}{\Z^{\op}\!}
\DeclareMathOperator{\spc}{Spc}
\DeclareMathOperator{\spec}{Spec}
\DeclareMathOperator{\spch}{Spc^{\mathrm{h}}}
\DeclareMathOperator{\spech}{Spec^{\mathrm{h}}}
\DeclareMathOperator{\supp}{supp}

\providecommand{\A}{\ensuremath{\mathcal A}}
\providecommand{\Cat}{\ensuremath{{\mathcal Cat}}}
\providecommand{\Catp}{\ensuremath{2\text{-}{\mathcal Cat}}}
\providecommand{\cf}{cf.\ }
\providecommand{\eg}{e.g.\ }
\providecommand{\id}{\ensuremath{\mathrm{id}}}
\providecommand{\ie}{i.e.\ }

\providecommand{\gr}{\ensuremath{\mathrm{gr}}}

\providecommand{\I}{\ensuremath{\mathcal I}}
\providecommand{\Kac}{\ensuremath{{\mathcal K}_{\mathrm{ac}}}}
\providecommand{\loccit}{loc.\,cit.\ }
\DeclareSymbolFont{bbold}{U}{bbold}{m}{n}
\DeclareSymbolFontAlphabet{\mathbbold}{bbold}
\providecommand{\one}{\ensuremath{\mathbbold{1}}}

\providecommand{\op}{\ensuremath{\mathrm{op}}}
\providecommand{\prmod}[1]{\ensuremath{\mathfrak{m}_{#1}}}
\providecommand{\prrat}{\ensuremath{\mathfrak{m}_{0}}}
\providecommand{\pret}[1]{\ensuremath{\mathfrak{e}_{#1}}}
\providecommand{\Q}{\ensuremath{\mathbb{Q}}}
\providecommand{\R}{\ensuremath{\mathcal R}}

\providecommand{\seqtimes}{\ensuremath{\otimes_{\Z^{\mathrm{op}}}}}
\providecommand{\T}{\ensuremath{\mathcal T}}

\providecommand{\ttCat}{\ensuremath{tt\Cat}}
\providecommand{\ttCatp}{\ensuremath{2\text{-}tt\Cat}}
\providecommand{\Z}{\ensuremath{\mathbb{Z}}}

\DeclarePairedDelimiter{\delimpar}{(}{)}
\newcommand{\prns}[2][0]{%
  \ifcase#1\relax
    \delimpar{#2}\or
    \delimpar[\big]{#2}\or
    \delimpar[\Big]{#2}\or
    \delimpar[\bigg]{#2}\or
    \delimpar[\Bigg]{#2}
  \else
    \delimpar*{#2}
  \fi}

\makeatletter
\def\scripts#1#2#3{\def\scripts@{\prns[#1]{#3}}\def\scripts@@{#2}\def\scripts@@@{#2}\@ifnextchar^\@sup\@nsup}
\def\@sup^#1{\def\scripts@@@{\scripts@@^{#1}}\@ifnextchar_\@supsub{\scripts@@@\scripts@}}
\def\@supsub_#1{\scripts@@@_{#1}\scripts@}
\def\@nsup{\@ifnextchar_{\@sub}{\scripts@@@\scripts@}}
\def\@sub_#1{\def\scripts@@@{\scripts@@_{#1}}\@ifnextchar^\@subsup{\scripts@@@\scripts@}}
\def\@subsup^#1{\scripts@@@^{#1}\scripts@}
\makeatother

\newcommand{\C}[2][0]{\scripts{#1}{\operatorname{\mathcal{C}}}{#2}}
\newcommand{\D}[2][0]{\scripts{#1}{\operatorname{\mathcal{D}}}{#2}}
\newcommand{\K}[2][0]{\scripts{#1}{\operatorname{\mathcal{K}}}{#2}}

\begin{document}
\title{tt-geometry of filtered modules}
\author{Martin Gallauer}
\begin{abstract}
  We compute the tensor triangular spectrum of perfect complexes of filtered modules over a commutative ring, and deduce a classification of the thick tensor ideals. We give two proofs: one by reducing to perfect complexes of graded modules which have already been studied in the literature \cite{ambrogio-stevenson:graded-ring,ambrogio-stevenson:tt-comparison-2-rings}, and one more direct for which we develop some useful tools.
\end{abstract}
\maketitle

\tableofcontents
\section{Introduction}
\label{sec:intro}

One of the age-old problems mathematicians engage in is to classify their objects of study, up to an appropriate equivalence relation. In contexts in which the domain is organized in a category with compatible tensor and triangulated structure (we call this a \emph{tt-category}) it is natural to view objects as equivalent when they can be constructed from each other using sums, extensions, translations, tensor product etc., in other words, using the tensor and triangulated structure alone. This can be made precise by saying that the objects generate the same thick tensor ideal (or, \emph{tt-ideal}) in the tt-category. This sort of classification is precisely what tt-geometry, as developed by Balmer, achieves. To a (small) tt-category $\T$ it associates a topological space $\spc(\T)$ called the \emph{tt-spectrum} of $\T$ which, via its Thomason subsets, classifies the (radical) tt-ideals of~$\T$. A number of classical mathematical domains have in the meantime been studied through the lens of tt-geometry; we refer to~\cite{balmer:icm} for an overview of the basic theory, its early successes and applications.

One type of context which does not seem to have received any attention so far arises from filtered objects. Examples pertinent to tt-geometry abound: filtrations by the weight in algebraic geometry induce filtrations on cohomology theories, giving rise to filtered vector spaces, representations or motives; (mixed) Hodge theory involves bifiltered vector spaces; filtrations by the order of a differential operator play an important role in the theory of ${\mathcal D}$-modules.

In this note, we take first steps in the study of filtered objects through the lens of tt-geometry by focusing on a particularly interesting case whose unfiltered analogue is well-understood. Namely, we give a complete account of the tt-geometry of filtered modules. This is already enough to say something interesting about certain motives, as we explain at the end of this introduction. To describe our results in more detail, let us recall the analogous situation for modules.

Let $R$ be a ring, assumed commutative and with unit. Its derived category $\D{R}$ is a tt-category which moreover is compactly generated, and the compact objects coincide with the rigid (or, strongly dualizable) objects, which are also called perfect complexes. These are (up to isomorphism in the derived category) the bounded complexes of finitely generated projective $R$-modules. The full subcategory $\per(R)$ of perfect complexes inherits the structure of a (small) tt-category, and the Hopkins-Neeman-Thomason classification of its thick subcategories can be interpreted as the statement that the tt-spectrum $\spc(\per(R))$ is precisely the Zariski spectrum~$\spec(R)$. In this particular case, thick subcategories are the same as tt-ideals so that this result indeed classifies perfect complexes up to the triangulated and tensor structure available.

In this note we will replicate these results for filtered $R$-modules. Its derived category $\D{\FMd(R)}$ is a tt-category which moreover is compactly generated, and the compact objects coincide with the rigid objects. We characterize these ``perfect complexes'' as bounded complexes of ``finitely generated projective'' objects in the category $\FMd(R)$ of filtered $R$-modules.\footnote{In the body of the text these are rather called \emph{split finite projective} for reasons which will become apparent when they are introduced.} The full subcategory $\fperalt$ of perfect complexes inherits the structure of a (small) tt-category. For a regular ring $R$ this is precisely the filtered derived category of $R$ in the sense first studied by Illusie in~\cite{illusie:cotangent-complex-I}, and for general rings it is a full subcategory. Our main theorem computes the tt-spectrum of this tt-category.
{
\renewcommand{\thethm}{\ref{main-thm}}
\begin{thm-intro}
  The tt-spectrum of $\fperalt$ is canonically isomorphic to the homogeneous Zariski spectrum $\spech(R[\beta])$ of the polynomial ring in one variable. In particular, the underlying topological space contains two copies of $\spec(R)$, connected by specialization. Schematically:
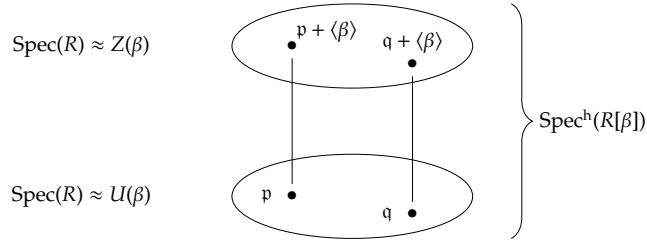
\begin{center}
\scalebox{0.8}{  \begin{tikzpicture}
    \node at (-4.5,0) {$\mathrm{Spec}(R)\approx U(\beta)$};
    \node at (-4.5,2.5) {$\mathrm{Spec}(R)\approx Z(\beta)$};
    \draw (0,0) ellipse (2cm and .7cm);
    \draw (0,2.5) ellipse (2cm and .7cm);
    \node[label=left:$\mathfrak{p}$] (p) at (-1,0) {$\bullet$};
    \node[label={[label distance=-.35cm]35:$\mathfrak{p}+\langle\beta\rangle$}] (pb) at (-1,2.5) {$\bullet$};
    \draw[-] (p.north) to (pb.south);
    \node[label=left:$\mathfrak{q}$] (q) at (1,-.3) {$\bullet$};
    \node[label={[label distance=-.15cm]90:$\mathfrak{q}+\langle\beta\rangle$}] (qb) at (1,2.2) {$\bullet$};
    \draw[-] (q.north) to (qb.south);
    \draw [decorate,decoration={brace,amplitude=10pt,mirror,raise=4pt},yshift=0pt] (2.5,-0.7) -- (2.5,3.2) node [black,midway,xshift=1.5cm] {$\spech(R[\beta])$};  \end{tikzpicture}
}\end{center}
\end{thm-intro}
\addtocounter{thm}{-1}
}
As a consequence we are able to classify the tt-ideals in~$\fperalt$. To state it precisely notice that we may associate to any filtered $R$-module $M$ its underlying $R$-module $\pi(M)$ as well as the $R$-module of its graded pieces~$\gr(M)$. These induce two tt-functors~$\fperalt\to\per(R)$. Also, recall that the \emph{support} of an object $M\in\per(R)$, denoted by $\supp(M)$, is the set of primes in $\spc(\per(R))=\spec(R)$ which do not contain $M$. This is extended to a set ${\mathcal E}$ of objects by taking the union of the supports of its elements: $\supp({\mathcal E}):=\cup_{M\in{\mathcal E}}\supp(M)$. Conversely, starting with a set of primes $Y\subset\spec(R)$, we define ${\mathcal K}_{Y}:=\{M\in\per(R)\mid \supp(M)\subset Y\}$.

{
\renewcommand{\thethm}{\ref{classification-tt-ideals}}
\begin{cor}
  There is an inclusion preserving bijection
  \begin{align*}
    \left\{\Pi\subset \Gamma\mid \Pi,\Gamma\subset\spec(R)\text{ Thomason subsets}
    \right\}&\longleftrightarrow
              \left\{
              \text{tt-ideals in }\fperalt
              \right\} \\
     (\Pi\subset \Gamma)&\longmapsto \pi^{-1}({\mathcal K}_{\Pi})\cap\gr^{-1}({\mathcal K}_{\Gamma})\\
    \left(
    \supp(\pi{\mathcal J})\subset\supp(\gr{\mathcal J})
    \right)&\longmapsfrom {\mathcal J}
  \end{align*}
\end{cor}
\addtocounter{thm}{-1}
}
Clearly, an important role is played by the element $\beta$ appearing in the Theorem. It can be interpreted as the following morphism of filtered $R$-modules. Let $R(0)$ be the module $R$ placed in filtration degree 0, while $R(1)$ is $R$ placed in degree 1 (our filtrations are by convention decreasing), and $\beta:R(0)\to R(1)$ is the identity on the underlying modules:\footnote{We call this element $\beta$ in view of the intended application described at the end of this introduction. In the context of motives considered there, $\beta$ is the ``Bott element'' of~\cite{haesemeyer-hornbostel:bott}.}
\begin{equation*}
  \xymatrix{R(1):&\cdots\ar@{}[r]|-*[@]{=}&0\ar@{}[r]|-*[@]{\subset}&R\ar@{}[r]|-*[@]{=}&R\ar@{}[r]|-*[@]{=}&\cdots\\
   R(0):\ar[u]_{\beta}& \cdots\ar@{}[r]|-*[@]{=}&0\ar[u]\ar@{}[r]|-*[@]{=}&0\ar[u]\ar@{}[r]|-*[@]{\subset}&R\ar[u]_{\id}\ar@{}[r]|-*[@]{=}&\cdots}
\end{equation*}

Note that $\beta$ has trivial kernel and cokernel but is not an isomorphism, witnessing the fact that the category of filtered modules is \emph{not} abelian. We will give two proofs of \cref{main-thm}, the first of which relies on ``abelianizing'' the category. It is observed in \cite{Schneiders:quasi-abelian} that the derived category of filtered modules is canonically identified with the derived category of graded $R[\beta]$-modules. And the tt-geometry of graded modules has been studied in \cite{ambrogio-stevenson:graded-ring,ambrogio-stevenson:tt-comparison-2-rings}. Together these two results provide a short proof of \cref{main-thm}, but in view of future studies of filtered objects in more general abelian tensor categories we thought it might be worthwile to study filtered modules in more detail and in their own right. For the second proof we will use the abelianization only minimally to construct the category of perfect complexes of filtered modules (\cref{sec:fmod,sec:fmod-derived}). The computation of the tt-geometry stays within the realm of filtered modules, as we now proceed to explain.

As mentioned above, forgetting the filtration and taking the associated graded of a filtered $R$-module gives rise to two tt-functors. It is not difficult to show that $\spc(\pi)$ and $\spc(\gr)$ are injective with disjoint images~(\cref{sec:main-thm}). The challenge is in proving that they are jointly surjective - more precisely, proving that the images of $\spc(\pi)$ and $\spc(\gr)$ are exactly the two copies of $\spec(R)$ in the picture above. As suggested by this then, and as we will prove, inverting $\beta$ (in a categorical sense) amounts to passing from filtered to unfiltered $R$-modules, while killing $\beta$ amounts to passing to the associated graded.

We prove surjectivity first for $R$ a noetherian ring, by reducing to the local case, using some general results we establish on central localization (\cref{sec:localization}), extending the discussion in~\cite{balmer:sss}. In the local noetherian case, the maximal ideal is ``generated by non-zerodivisors and nilpotent elements'' (more precisely, it admits a system of parameters); we will study how killing such elements affects the tt-spectrum (\cref{sec:reduction}) which allows us to decrease the Krull dimension of $R$ one by one until we reach the case of $R$ a field.

Although the category of filtered modules is not abelian, it has the structure of a \emph{quasi}-abelian category, and we will use the results of Schneiders on the derived category of a quasi-abelian category, in particular the existence of two t-structures~\cite{Schneiders:quasi-abelian}, to deal with the case of a field~(\cref{sec:field}). In fact, the category of filtered vector spaces can reasonably be called a \emph{semisimple} quasi-abelian category, and we will prove in general that the t-structures in that case are hereditary. With this fact it is then possible to deduce the theorem in the case of a field.

Finally, we will reduce the case of arbitrary rings to noetherian rings (\cref{sec:continuity}) by proving in general that tt-spectra are \emph{continuous}, that is for filtered colimits of tt-categories one has a canonical homeomorphism
\begin{equation*}
  \spc(\varinjlim_{i}\T_{i})\xrightarrow{\sim}\varprojlim_{i}\spc(\T_{i}).
\end{equation*}
In fact, we will prove a more general statement which we believe will be useful in other studies of tt-geometry as well, because it often allows to reduce the tt-geometry of ``infinite objects'' to the tt-geometry of ``finite objects''. For example, it shows immediately that the noetherianity assumption in the results of~\cite{ambrogio-stevenson:graded-ring} is superfluous, arguably simplifying the proof given for this observation in~\cite{ambrogio-stevenson:tt-comparison-2-rings}.

We mentioned above that one of our motivations for studying the questions discussed in this note lies in the theory of motives. Let us therefore give the following application. We are able to describe completely the spectrum of the triangulated category of Tate motives over the algebraic numbers with integer coefficients,~$\dtm(\overline{\Q},\Z)$. (Previously, only the rational part $\dtm(\overline{\Q},\Q)$ was known.)
\begin{thm*}
The tt-spectrum of $\dtm(\overline{\Q},\Z)$ consists of the following primes, with specialization relations as indicated by the lines going upward.
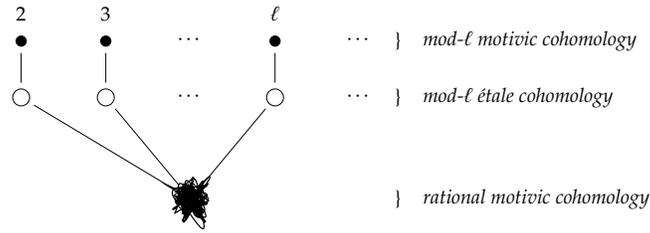
\begin{center}
  \vspace{0.2cm} \scalebox{0.75}{    \begin{tikzpicture}[shorten >=3pt,shorten <=3pt, level
      1/.style={level distance=1.8cm}, level 2/.style={level
        distance=1cm}, every node/.style={inner sep=2pt},
      mod/.style={circle, fill=black,}, et/.style={circle, fill=black}, root/.style={circle, fill=black}, no/.style={edge from
        parent/.style={}} ]
      \node (root) [root,label={[label distance=8mm]right: $\prrat$}] {} [grow'=up] child {node [et] {} child
        {node [mod] {}} } child {node [et] {} child {node [mod] {}} }
      child[no] {node {$\cdots$} child {node {$\cdots$}} } child {node
        [et,label=right: {$\pret{\ell}$}] {} child {node [mod,label=right: {$\prmod{\ell}$}] {}} }child[no] {node {$\cdots$} child {node {$\cdots$}} } ;
      \begin{scope}[every node/.style={right}]
        \xdef\level{root} 
        \def\rightmostnode{root-3-1} 
        \foreach \text in {rational motivic cohomology, mod-{$\ell$} \'{e}tale
          cohomology, mod-{$\ell$} motivic cohomology} { \path (\level
           -|\rightmostnode) ++(35mm,0) node{$\}$} ++(5mm,0) node
           {\text}; \xdef\level{\level-1} }
       \end{scope}
       \coordinate (br) at (0,-4mm);
     \end{tikzpicture}} \vspace{0.1cm}
\end{center}

Here, $\ell$ runs through all prime numbers, and the primes are defined by the vanishing of the cohomology theories as indicated on the right. Moreover, the proper closed subsets are precisely the finite subsets stable under specialization.
\end{thm*}
As a consequence, we are able to classify the thick tensor ideals of~$\dtm(\overline{\Q},\Z)$. This Theorem and related results are proved in a separate paper~\cite{gallauer:tt-dtm-algclosed}.

\section*{Acknowledgment}

I would like to thank Paul Balmer for his interest in this note, and his critical input on an earlier version. I am also grateful to the anonymous referee whose suggestions led to improved exposition.

\section*{Conventions}
\label{sec:conventions}

A symmetric, unital monoidal structure on a category is called \emph{tensor structure} if the category is additive, and if the monoidal product is additive in each variable separately. We also call these data simply a \emph{tensor category}. A \emph{tensor functor} between tensor categories is a strong, unital, symmetric monoidal additive functor.

Our conventions regarding tensor triangular geometry mostly follow those of~\cite{balmer:sss}. A \emph{tensor triangulated category} (or \emph{tt-category} for short) is a triangulated category with a compatible (symmetric, unital) tensor structure. Typically, one assumes that the category is (essentially) small. If not specified otherwise, the tensor product is denoted by $\otimes$ and the unit by~$\one$. A \emph{tt-functor} is an exact tensor functor between tt-categories.

A \emph{tt-ideal} in a tt-category $\T$ is a thick subcategory $\I\subset \T$ such that~$\T\otimes \I\subset\I$. If $S$ is a set of objects in $\T$ we denote by $\langle S\rangle$ the tt-ideal generated by~$S$. To a small rigid tt-category $\T$ one associates a locally ringed space $\spec(\T)$, called the \emph{tt-spectrum of $\T$}, whose underlying topological space is denoted by $\spc(\T)$. It consists of \emph{prime ideals} in $\T$, \ie proper tt-ideals $\I$ such that $a\otimes b\in\I$ implies $a\in\I$ or~$b\in\I$. (The underlying topological space $\spc(\T)$ is defined even if $\T$ is not rigid.)

All rings are commutative with unit, and morphisms of rings are unital. For $R$ a ring, we denote by $\spec(R)$ the Zariski spectrum of $R$ (considered as a locally ringed space) whereas $\spc(R)$ denotes its underlying topological space (as for the tt-spectrum). We adopt similar conventions regarding graded rings $R$: they are commutative in a general graded sense~\cite[3.4]{balmer:sss}, and possess a unit. $\spech(R)$ denotes the homogeneous Zariski spectrum with underlying topological space~$\spch(R)$.

As a general rule, canonical isomorphisms in categories are typically written as equalities.

\section{Category of filtered modules}
\label{sec:fmod}
In this section we describe filtered modules from a slightly nonstandard perspective which will be useful in the sequel. Hereby we follow the treatment in~\cite{schapira-schneiders:filtered}. The idea is to embed the (non-abelian) category of filtered modules into its \emph{abelianization}, the category of presheaves of modules on the poset $\Z$. From this embedding we deduce a number of properties of the category of filtered modules. Much of the discussion in this section applies more generally to filtered objects in suitable abelian tensor categories.

Fix a commutative ring with unit~$R$. Denote by $\Md(R)$ the abelian category of $R$-modules, with its canonical tensor structure. We view $\Z$ as a monoidal category where
\begin{equation*}
\hom(m,n)=
\begin{cases}
  \{\ast\}&:m\leq n\\
  \emptyset&:m>n
\end{cases}
\end{equation*}
and~$m\otimes n=m+n$. The Day convolution product then induces a tensor structure on the category of presheaves on $\Z$ with values in $\Md(R)$ which we denote by~$\seq R$. Explicitly, an object $a$ of $\seq{R}$ is an infinite sequence of morphisms in $\Md(R)$
\begin{equation}\label{sequence}
  \cdots\to a_{n+1}\xrightarrow{a_{n,n+1}} a_{n}\xrightarrow{a_{n-1,n}} a_{n-1}\to\cdots,
\end{equation}
and the tensor product of two such objects $a$ and $b$ is described by
\begin{equation*}
  (a\seqtimes b)_{n}=\colim_{p+q\geq n}a_{p}\otimes_{R} b_{q}.
\end{equation*}
Let $M$ be an $R$-module and~$n\in\Z$. The associated presheaf $\oplus_{\hom_{\Z}(-,n)}M$ is denoted by~$M(n)$. It is the object
\begin{equation*}
  \cdots\to 0\to 0\to M\xrightarrow{\id}M\xrightarrow{\id} M\to\cdots
\end{equation*}
with the first $M$ in degree~$n$. Via the association $\sigma_{0}:M\mapsto M(0)$ we view $\Md(R)$ as a full subcategory of~$\seq{R}$. For any object $a\in\seq{R}$ and $n\in\Z$ we denote by $a(n)$ the tensor product $a\otimes R(n)$, and we call it the $n$th twist of~$a$. Explicitly, this is the sequence of \cref{sequence} shifted to the left by $n$ places, \ie $a(n)_{m}=a_{m-n}$.

The category $\seq{R}$ is $R$-linear Grothendieck abelian, and the monoidal structure is closed. Explicitly, the internal hom of $a,b\in\seq{R}$ is given by
\begin{equation*}
  \ihom(a,b)_{n}=\hom_{\seq{R}}(a(n),b).
\end{equation*}

Here is another way of thinking about $\seq{R}$. Let $a\in\seq{R}$ be a presheaf of $R$-modules. Associate to it the graded $R[\beta]$-module $\oplus_{n\in\Z}a_{n}$ with $\beta$ acting by $\beta:a\to a(1)$, \ie in degree $n$ by $a_{n-1,n}:a_{n}\to a_{n-1}$. In particular, $\beta$ is assumed to have degree -1. Conversely, given a graded $R[\beta]$-module $\oplus_{n\in\Z}M_{n}$, define a presheaf by $n\mapsto M_{n}$ and transition maps $\cdot\beta:M_{n}\to M_{n-1}$. This clearly establishes an isomorphism of categories $\seq{R}=\GMd(R[\beta])$, and it is not difficult to see that the isomorphism is compatible with the tensor structures on both sides.

\begin{dfn}
  \begin{enumerate}
  \item A \emph{filtered $R$-module} is an object $a\in \seq{R}$ such that $a_{n,n+1}$ is a monomorphism for all~$n\in\Z$. The full subcategory of filtered $R$-modules in $\seq{R}$ is denoted by~$\FMd(R)$.
  \item A \emph{finitely filtered $R$-module} is a filtered $R$-module $a$ such that $a_{n,n+1}$ is an isomorphism for almost all~$n$.
  \item A filtered $R$-module $a$ is \emph{separated} if~$\cap_{n\in\Z}a_{n}=0$.
  \end{enumerate}
\end{dfn}
For a filtered $R$-module $a$ we denote the ``underlying'' $R$-module~$\varinjlim_{n\to -\infty}a_{n}$ by $\pi(a)$. This clearly defines a functor $\pi:\FMd(R)\to\Md(R)$ which ``forgets the filtration''. In this way we recover the more classical perspective on filtrations: an $R$-module $\pi(a)$ together with a (decreasing, exhaustive) filtration $(a_{n})_{n\in\Z}$; a morphism $f:a\to b$ of filtered $R$-modules $a,b$ is an $R$-linear morphism $\pi(a)\to \pi(b)$ compatible with the filtration.

To a filtered $R$-module $a$ one can associate its ($\Z$-)graded $R$-module whose $n$th graded piece is~$\coker(a_{n,n+1})=a_{n}/a_{n+1}$. This clearly defines a functor~$\gr_{\bullet}:\FMd(R)\to\GMd(R)=\prod_{n\in\Z}\Md(R)$.

The following observation is simple but very useful.

\begin{lem}[{\cite[3.5]{schapira-schneiders:filtered}}]\label{fmod-reflective}
  The inclusion $\iota:\FMd(R)\to\seq{R}$ admits a left adjoint $\kappa:\seq{R}\to\FMd(R)$ given by
  \begin{equation*}
    \kappa(a)_{n}=\img(a_{n}\to\varinjlim_{m\to -\infty}a_{m})
  \end{equation*}
  and the canonical transition maps.
\end{lem}

It follows from \cref{fmod-reflective} that $\FMd(R)$ is complete and cocomplete. Limits, filtered colimits and direct sums are computed in $\seq{R}$ while pushouts are computed by applying the reflector $\kappa$ to the pushout in~$\seq{R}$. (The statement about limits and pushouts is formal, while the rest stems from the fact that filtered colimits and direct sums are exact in $\Md(R)$.) In particular, $\FMd(R)$ is additive and has kernels and cokernels. However, it is not an abelian category as witnessed by the morphism
\begin{equation}
  \label{beta}
  \beta:R(0)\to R(1)
\end{equation}
induced by the map $0\to 1$ in $\Z$ through the Yoneda embedding: both $\ker(\beta)$ and $\coker(\beta)$ are 0 but $\beta$ is not an isomorphism. It is an example of a \emph{non-strict} morphism. (A morphism $f:a\to b$ is called \emph{strict} if the canonical morphism $\coimg(f)\to\img(f)$ is an isomorphism, or equivalently if $\img(\pi(f))\cap b_{n}=\img(f_{n})$ for all~$n\in\Z$.) However, one can easily check that strict monomorphisms and strict epimorphisms in $\FMd(R)$ are preserved by pushouts and pullbacks, respectively~\cite[3.9]{schapira-schneiders:filtered}. In other words, $\FMd(R)$ is a \emph{quasi-}abelian category (we will use~\cite{Schneiders:quasi-abelian} as a reference for the basic theory of quasi-abelian categories).

An object $a$ in a quasi-abelian category is called \emph{projective} if $\hom(a,-)$ takes strict epimorphisms to surjections. (Note that this convention differs from the categorical notion of a projective object!) For example, for a projective $R$-module $M$ and $n\in\Z$ the object $M(n)$ is projective since~$\hom_{\FMd(R)}(M(n),-)= \hom_{R}(M,(-)_{n})$.
\begin{lem}[{\cite[3.1.8]{Schneiders:quasi-abelian}}]\label{fmod-enough-projectives}
  For any $a\in\FMd(R)$, the canonical morphism
  \begin{equation}
    \oplus_{n\in\Z}\oplus_{x\in a_{n}}R(n)\to a
  \end{equation}
  is a strict epimorphism with projective domain. In particular, the quasi-abelian category $\FMd(R)$ has enough projectives.
\end{lem}

Let us denote by $\sigma:\GMd(R)\to\FMd(R)$ the canonical functor which takes $(M_{n})_{n}$ to $\oplus_{n}M_{n}(n)$. A filtered $R$-module is called \emph{split} if it lies in the essential image of $\sigma$. Correspondingly we call a filtered $R$-module \emph{split free} (respectively, \emph{split projective}, \emph{split finite projective}) if it is (isomorphic to) the image of a free (respectively, projective, finite projective) graded $R$-module under $\sigma$. In other words, an object of the form $\oplus_{n}M_{n}(n)$ with $\oplus_{n}M_{n}$ free (respectively,  projective, finite projective). \cref{fmod-enough-projectives} shows that every object in $\FMd(R)$ admits a canonical split free resolution. 

It is clear that split projective objects are projective, and the converse is also true as we now prove (see \cref{fmod-projectives} below).

\begin{lem}\label{split-projective-idempotent}
  The full additive subcategory $\FP(R)$ of split projectives is idempotent complete. The same is true for the full additive subcategory $\fmd(R)$ of split finite projectives.
\end{lem}
\begin{proof}
  Let $f:a\xrightarrow{\sim} b\oplus c$ be an isomorphism, with $a$ split projective. Since $a$ is split, there is a canonical isomorphism $g:a\xrightarrow{\sim}\oplus_{n}\gr_{n}(a)(n)$, and we can define the following composition of isomorphisms:
  \begin{equation*}
    b\oplus c\xrightarrow[\sim]{f^{-1}}a\xrightarrow[\sim]{g}\oplus_{n}\gr_{n}(a)(n)
    \xrightarrow[\sim]{f}\oplus_{n}\gr_{n}(b\oplus c)(n)=
    \left(
      \oplus_{n}\gr_{n}(b)(n)
\right)\oplus 
\left(
  \oplus_{n}\gr_{n}(c)(n)
\right).
  \end{equation*}
It is easy to see that this induces an isomorphism $b\cong\oplus_{n}\gr_{n}(b)(n)$, and we also see that $\gr_{n}(b)$ is a direct summand of $\gr_{n}(a)$. In other words, $b$ is split projective as required. The same proof applies in the finite case.
\end{proof}

\begin{lem}\label{fmod-projectives}
  For a filtered $R$-module $a\in \FMd(R)$ the following are equivalent:
  \begin{enumerate}
  \item $a$ is projective.
  \item $a$ is split projective.
  \end{enumerate}
\end{lem}
\begin{proof}
  Let $a$ be projective. As remarked in \cref{fmod-enough-projectives}, there is a canonical strict epimorphism $b\to a$ with $b$ split free. By definition of projectivity, there is a section $a\to b$, and since $\FMd(R)$ has kernels and images, we deduce that $a$ is a direct summand of $b$. It therefore suffices to prove that every direct summand of a split free is split projective. This follows from \cref{split-projective-idempotent}.
\end{proof}

In general, due to the possibility of the tensor product in $\Md(R)$ not being exact, the tensor structure on $\seq{R}$ does not restrict to the subcategory~$\FMd(R)$. We can use the reflector $\kappa$ to rectify this: for $a,b\in\FMd(R)$, let
\begin{equation*}
  a\otimes b=\kappa(\iota(a)\seqtimes\iota(b)).
\end{equation*}
This defines a tensor structure on $\FMd(R)$.\footnote{This can be seen as a particular instance of \cite{day:reflection} due to the canonical isomorphisms
  \begin{equation*}
    \kappa(a\seqtimes \kappa(b))\xleftarrow{\sim}\kappa(a\seqtimes b)\xrightarrow{\sim}\kappa(\kappa(a)\seqtimes b)
  \end{equation*}
  for any $a,b\in\seq$.} It is clear that the internal hom on $\seq{R}$ restricts to a bifunctor on $\FMd(R)$, and it follows formally from \cref{fmod-reflective} that this bifunctor is the internal hom on~$\FMd(R)$.

Although we will in the sequel only use the implication (1)$\Rightarrow$(2) of the following result, it is satisfying to see these notions match up as they do in $\Md(R)$. Recall that an object $a$ in a category with filtered colimits is called \emph{finitely presented} if $\hom(a,-)$ commutes with these filtered colimits.
\begin{lem}\label{fmod-finite-projectives}
  For a filtered $R$-module $a\in\FMd(R)$ the following are equivalent:
  \begin{enumerate}
  \item $a$ is split finite projective.
  \item $a$ is rigid (or strongly dualizable).
  \item $a$ is finitely presented and projective.
  \end{enumerate}
\end{lem}
\begin{proof}
  Since $\sigma:\GMd(R)\to\FMd(R)$ is a tensor functor it preserves rigid objects. This shows the implication (1)$\Rightarrow$(2).

  For (2)$\Rightarrow$(3) notice that the unit $R(0)$ is both finitely presented and projective. The latter is clear, and the former is true as filtered colimits are computed in $\seq{R}$. The implication is now obtained from the identification
  \begin{equation*}
    \hom(a,-)=\hom(R(0),\ihom(a,R(0))\otimes -)
  \end{equation*}
  together with the fact that the tensor product preserves filtered colimits and strict epimorphisms.

  Finally for (3)$\Rightarrow$(1), we start with the identification $a=\oplus_{n}\gr_{n}(a)(n)$ with $\gr_{n}(a)$ projective $R$-modules, which exists by \cref{fmod-projectives}. Notice that the forgetful functor $\pi:\FMd(R)\to\Md(R)$ has a right adjoint $\Delta:\Md(R)\to\FMd(R)$ which takes an $R$-module to the same $R$-module with the constant filtration. It is clear that $\Delta$ commutes with filtered colimits so that
  \begin{equation*}
    \hom(\pi(a),\varinjlim -)=\hom(a,\varinjlim\Delta-)=\varinjlim\hom(a,\Delta -)=\varinjlim\hom(\pi(a),-)
  \end{equation*}
and hence $\pi(a)$ is a finitely presented $R$-module. We conclude that $a=\oplus\gr_{n}(a)(n)$ is split finite projective.
\end{proof}

\begin{cor}\label{fmod-projectives-tensor}
  \begin{enumerate}
  \item If $a\in\FMd(R)$ is projective then $a\otimes-$ preserves kernels of arbitrary morphisms.
  \item If $a,b\in\FMd(R)$ are projective then so is~$a\otimes b$.
  \end{enumerate}
\end{cor}
\begin{proof}
 Since the tensor product commutes with direct sums both statements follow from \cref{fmod-projectives}.
\end{proof}
\section{Derived category of filtered modules}
\label{sec:fmod-derived}
Quasi-abelian categories are examples of exact categories and can therefore be derived in the same way. However, the theory for quasi-abelian categories is more precise and we will exploit this fact starting in the current section. In the case of (separated, finitely) filtered $R$-modules we obtain what is classically known as the filtered derived category of~$R$. Some of its basic properties are established, a number of which are deduced from the relation with the derived category of $\seq{R}$.

For $\ast\in\{b,-,+,\emptyset\}$ we denote by $\C{\FMd(R)}^{\ast}$ the category of bounded (respectively bounded above, bounded below, unbounded) cochain complexes in $\FMd(R)$, and by $\K{\FMd(R)}^{\ast}$ the associated homotopy category. A complex
\begin{equation*}
  A:\quad \cdots\to A^{l-1}\xrightarrow{d^{l-1}}A^{l}\xrightarrow{d^{l}}A^{l+1}\to\cdots
\end{equation*}
is called \emph{strictly exact} if all differentials $d^{l}$ are strict, and the canonical morphism $\img(d^{l-1})\to\ker(d^{l})$ is an isomorphism for all~$l$. We note the following simple but useful fact.
\begin{lem}[{\cite[1]{sjodin:filtered-graded-modules}}]\label{strictly-exact}
  Let $A$ be a complex in $\FMd(R)$ and consider the following conditions:
  \begin{enumerate}
  \item $A$ is strictly exact;
  \item all its differentials $d^{l}$ are strict and the underlying complex $\pi(A)$ is exact;
  \item the associated graded complex $\gr_{\bullet}(A)$ is exact, \ie $\gr_{n}(A)$ is an exact complex for all~$n\in\Z$.
  \end{enumerate}
  We have (1)$\Leftrightarrow$(2)$\Rightarrow$(3), and if $A^{l}$ is finitely filtered and separated for all $l\in\Z$ then all conditions are equivalent.
\end{lem}
The class of strictly exact complexes forms a saturated null system $\Kac^{\ast}$~\cite[1.2.15]{Schneiders:quasi-abelian} and we set~$\D{\FMd(R)}^{\ast}=\K{\FMd(R)}^{\ast}/\Kac^{\ast}$. The canonical triangulated structure on $\K{\FMd(R)}^{\ast}$ induces a triangulated structure on~$\D{\FMd(R)}^{\ast}$. As follows from \cref{strictly-exact}, this definition is an extension of the classical ``filtered derived category'' considered in~\cite{illusie:cotangent-complex-I}. There, complexes are assumed to be (uniformly) finitely filtered separated, and the localization is with respect to filtered quasi-isomorphisms, \ie morphisms $f:A\to B$ of complexes such that $\gr_{n}(f)$ is a quasi-isomorphism of complexes of $R$-modules, for all~$n\in \Z$.

The functor $\iota:\FMd(R)\to\seq{R}$ clearly preserves strictly exact complexes (we say that $\iota$ is \emph{strictly exact}), hence it derives trivially to an exact functor of triangulated categories~$\iota:\D{\FMd(R)}^{\ast}\to\D{\seq{R}}^{\ast}$.
\begin{pro}[{\cite[3.16]{schapira-schneiders:filtered}}]\label{fmod-seq-equivalence}
  The functor $\iota:\D{\FMd(R)}^{\ast}\to\D{\seq{R}}^{\ast}$ is an equivalence of categories. Its quasi-inverse is given by the left derived functor of~$\kappa$.
\end{pro}
Explicitly, $\dL\kappa$ may be computed using the ``Rees functor'' $\lambda:\seq{R}\to\FMd(R)$ which takes $a\in\seq{R}$ to the filtered $R$-module $\lambda(a)$ with
\begin{equation*}
  \lambda(a)_{n}=\oplus_{m\geq n}a_{m}
\end{equation*}
and the obvious inclusions as transition maps~\cite[3.12]{schapira-schneiders:filtered}. It comes with a canonical epimorphism $\varepsilon_{a}:\iota\lambda(a)\to a$ and since $\FMd(R)$ is closed under subobjects in $\seq{R}$, objects in $\seq{R}$ admit an additively functorial two-term resolution by objects in~$\FMd(R)$. Thus a complex $A$ in $\seq{R}$ is replaced by the cone of $\ker(\varepsilon_{A})\to \iota\lambda(A)$ which is a complex in $\FMd(R)$ and computes~$\dL\kappa(A)$.

The tensor product $\seqtimes$ on $\seq{R}$ can be left derived and yields
  \begin{equation*}
    \seqtimes^{\dL}:\D{\seq{R}}^{\ast}\times\D{\seq{R}}^{\ast}\to\D{\seq{R}}^{\ast}
  \end{equation*}
  for~$\ast\in\{-,\emptyset\}$. This follows for example from~\cite[2.3]{cisinski-deglise:homalg-grothendieck-cats} (where the descent structure is given by~$({\mathcal G}=\{R(n)\mid n\in\Z\}, {\mathcal H}=\{0\})$).
\begin{lem}\label{fmod-seq-equivalence-tensor}
  The tensor product on $\FMd(R)$ induces a left-derived tensor product
  \begin{equation*}
    \otimes^{\dL}:\D{\FMd(R)}^{\ast}\times\D{\FMd(R)}^{\ast}\to\D{\FMd(R)}^{\ast}
  \end{equation*}
  where~$\ast\in\{-,\emptyset\}$. Moreover, the equivalence of \cref{fmod-seq-equivalence} is compatible with the derived tensor products.
\end{lem}
\begin{proof}
  Recall that the tensor product was defined as~$\kappa\circ\seqtimes\circ(\iota\times\iota)$. Therefore the left-derived tensor product is given by
  \begin{equation*}
    \otimes^{\dL}=\dL\kappa\circ\seqtimes^{\dL}\circ(\iota\times\iota).
  \end{equation*}
  The second statement is then clear.
\end{proof}

\begin{cor}\label{fmod-der-compact}
  The triangulated category $\D{\FMd(R)}$ is compactly generated. For an object $A\in\D{\FMd(R)}$ the following are equivalent:
  \begin{enumerate}
  \item $A$ is compact.
  \item $A$ is rigid.
  \item $A$ is isomorphic to a bounded complex of split finite projectives.
  \end{enumerate}
\end{cor}
\begin{proof}
  It is easy to see \cite[3.20]{choudhury-gallauer:dg-hty} that the set $\{R(n)\mid n\in\Z\}$ compactly generates $\D{\seq{R}}$. The first statement therefore follows from \cref{fmod-seq-equivalence}. As is true in general \cite[2.2]{neeman:localizing-smashing}, the compact objects span precisely the thick subcategory generated by these generators $R(n)$. From this we see immediately that (3) implies (1). The converse implication follows from \cref{fper-homotopy-derived} below.

That (3) implies (2) is easy to see, using \cref{fmod-finite-projectives}. Finally that (2) implies (1) follows formally from the tensor unit being compact (\cf the proof of \cref{fmod-finite-projectives}).
\end{proof}

We denote by $\fper(R)$ the full subcategory of compact objects in~$\D{\FMd(R)}$. Its objects are also called \emph{perfect filtered complexes}. Note that this is an idempotent complete, rigid tt-category. We denote the tensor product on $\fper(R)$ simply by~$\otimes$. Recall that $\fmd(R)$ denotes the additive category of split finite projective filtered $R$-modules.

\begin{cor}\label{fper-homotopy-derived}
  The canonical functor $\K{\fmd(R)}^{b}\to \D{\FMd(R)}$ induces an equivalence of tt-categories
  \begin{equation*}
    \K{\fmd(R)}^{b}\xrightarrow{\sim}\fper(R).
  \end{equation*}
\end{cor}
\begin{proof}
  The fact that the image of the functor is contained in $\fper(R)$ was proved in \cref{fmod-der-compact}. It therefore makes sense to consider the following square of canonical exact functors
  \begin{equation*}
    \xymatrix{\K{\fmd(R)}^{b}\ar[r]\ar[d]&\fper(R)\ar[d]\\
      \K{\FP(R)}^{-}\ar[r]&\D{\FMd(R)}^{-}}
  \end{equation*}
  The vertical arrows are the inclusions of full subcategories. (For the right vertical arrow this follows from \cite[11.7]{keller:derived-categories}.) Moreover, the bottom horizontal arrow is an equivalence, by \cite[1.3.22]{Schneiders:quasi-abelian} together with \cref{fmod-enough-projectives}. We conclude that the top horizontal arrow is fully faithful as well.

  Next, we notice that since $\fmd(R)$ is idempotent complete by \cref{split-projective-idempotent}, the same is true of its bounded homotopy category \cite[2.8]{Balmer-Schlichting}. It follows that the image of the top horizontal arrow is a thick subcategory containing $R(n)$, $n\in\Z$. As remarked in the proof of \cref{fmod-der-compact}, this implies essential surjectivity.

  As tensoring with a split finite projective is strictly exact, by \cref{fmod-projectives-tensor}, the same is true for objects in $\K{\fmd(R)}^{b}$. It is then clear that the equivalence just established preserves the tensor product.
\end{proof}

For future reference we record the following simple fact.
\begin{lem}
  Let $\mathcal{J}\subset\fper(R)$ be a thick subcategory. Then the following are equivalent.
  \begin{enumerate}
  \item ${\mathcal J}$ is a tt-ideal.
  \item ${\mathcal J}$ is closed under $R(n)\otimes -$,~$n\in\Z$.
  \end{enumerate}
\end{lem}
\begin{proof}
  As remarked in the proof of \cref{fmod-der-compact}, the category of filtered complexes $\fper(R)$ is generated as a thick subcategory by $R(n)$,~$n\in\Z$. Thus (2) implies~(1):
  \begin{equation*}
    {\mathcal J}\otimes\fper(R)={\mathcal J}\otimes\langle R(n)\mid n\in\Z\rangle^{\mathrm{thick}}\subset{\mathcal J}.
  \end{equation*}
  The converse is trivial.
\end{proof}

Let us discuss the derived analogues of the functors $\pi$ and $\gr_{\bullet}$ introduced earlier.
\begin{lem}\label{pi-derived}
  The functor $\pi:\FMd(R)\to \Md(R)$ is strictly exact and derives trivially to a tt-functor $\pi:\D{\FMd(R)}\to\D{R}$. The latter preserves compact objects and restricts to a tt-functor
  \begin{equation*}
    \pi:\fper(R)\to\per(R),
  \end{equation*}
  where $\per(R)$ denotes the category of perfect complexes over $R$, \ie the compact objects in~$\D{R}$.
\end{lem}
\begin{proof}
  The first statement follows from \cref{strictly-exact}. The functor $\pi$ being tensor, it preserves rigid objects and the second statement follows from \cref{fmod-der-compact}.
\end{proof}

\begin{lem}\label{gr-derived}
  The functor $\gr_{\bullet}:\FMd(R)\to\GMd(R)$ is strictly exact and derives trivially to an exact functor $\gr_{\bullet}:\D{\FMd(R)}\to\D{\GMd(R)}$. The latter preserves compact objects and induces a conservative tt-functor
  \begin{equation*}
    \gr:=\oplus\gr_{\bullet}:\fper(R)\to\per(R).
  \end{equation*}
\end{lem}
\begin{proof}
  That $\gr_{\bullet}$ is strictly exact is \cref{strictly-exact}. It follows that $\gr_{\bullet}$ derives trivially to give an exact functor $\gr_{\bullet}:\D{\FMd(R)}\to\D{\GMd(R)}$. For each $n$, $\gr_{n}$ clearly sends perfect filtered complexes to perfect complexes, \ie $\gr_{\bullet}$ preserves compact objects (by \cref{fmod-der-compact}).

The functor $\oplus:\GMd(R)\to\Md(R)$ is strictly exact (in fact, it preserves arbitrary kernels and cokernels) and hence derives trivially as well to give a tt-functor which preserves compact objects.

There is a canonical natural transformation (on the underived level)
\begin{equation*}
  \gr_{\bullet}\otimes\gr_{\bullet}\to\gr_{\bullet}\circ\otimes
\end{equation*}
endowing $\gr_{\bullet}$ with the structure of a unital lax monoidal functor~\cite[3]{sjodin:filtered-graded-modules}. This natural transformation is easily seen to be an isomorphism for split finite projective filtered $R$-modules~\cite[12]{sjodin:filtered-graded-modules}. It follows that $\gr:\K{\fmd(R)}^{b}\to\K{\md(R)}^{b}$ is a tt-functor ($\md(R)$ is the category of finitely generated projective $R$-modules). Conservativity of this functor follows from \cref{strictly-exact}.
\end{proof}

Finally, notice that viewing $\Md(R)$ as a tensor subcategory of $\FMd(R)$ induces a section
\begin{equation*}
  \sigma_{0}:\per(R)\to\fper(R)
\end{equation*}
to both $\gr$ and~$\pi$.

\section{Main result}
\label{sec:main-thm}

The set of endomorphisms of the unit in a tt-category $\T$ is a (unital commutative) ring $\R_{\T}$, called the central ring of~$\T$. There is a canonical morphism of locally ringed spaces
\begin{equation*}
  \rho_{\T}:\spec(\T)\to\spec(\R_{\T})
\end{equation*}
comparing the tt-spectrum of $\T$ with the Zariski spectrum of its central ring, as explained in~\cite{balmer:sss}.

There is also a graded version of this construction. Given an invertible object $u\in\T$, it makes sense to consider the graded central ring with respect to $u$~(\cite[3.2]{balmer:sss}, see also \cref{sec:localization} for further discussion):
\begin{equation*}
  {\mathcal R}^{\bullet}_{\T,u}:=\hom_{\T}(\one,u^{\otimes\bullet}),\qquad \bullet\in\Z.
\end{equation*}
This is a unital $\epsilon$-commutative graded ring~\cite[3.3]{balmer:sss} and we can therefore consider its homogeneous spectrum. There is again a canonical morphism of locally ringed spaces
\begin{equation*}
  \rho^{\bullet}_{\T,u}:\spec(\T)\to\spech({\mathcal R}^{\bullet}_{\T,u}).
\end{equation*}
The inclusion $\R_{\T}\to \R_{\T,u}^{\bullet}$ as the degree 0 part provides a factorization~$\rho_{\T}=( R_{\T}\cap -)\circ\rho_{\T,u}^{\bullet}$.

Let us specialize to~$\T=\fper(R)$. The object $R(1)\in\fper(R)$ is clearly invertible and we define $\R_{R}^{\bullet}:=\R_{\fper(R),R(1)}^{\bullet}$, so that
\begin{equation*}
  {\mathcal R}^{\bullet}_{R}=\hom_{\fper(R)}(R(0),R(\bullet)),\qquad \bullet\in\Z.
\end{equation*}
Also,~$\rho^{\bullet}_{R}:=\rho_{\fper(R),R(1)}^{\bullet}$.

We are now in a position to state our main result.

\begin{thm}\label{main-thm}
  \mbox{}
  \begin{enumerate}
  \item The graded central ring ${\mathcal R}^{\bullet}_{R}$ is canonically isomorphic to the polynomial ring $R[\beta]$ where $\beta:R(0)\to R(1)$ as in \cref{beta} has degree 1.
  \item The morphism
    \begin{equation*}
      \rho^{\bullet}_{R}:\spec(\fper(R))\to\spech(R[\beta])
\end{equation*}
is an isomorphism of locally ringed spaces.
  \end{enumerate}
\end{thm}
The first part is immediate: by \cref{fper-homotopy-derived}, morphisms $R(0)\to R(n)$ in $\fper(R)$ may be computed in the homotopy category into which $\fmd(R)$ embeds fully faithfully. Using the Yoneda embedding we therefore find
\begin{align*}
  \hom_{\fper(R)}(R(0),R(n))&=\hom_{\K{\fmd(R)}^{\mathrm{b}}}(R(0),R(n))\\
  &=\hom_{\fmd(R)}(R(0),R(n))\\
  &=\oplus_{\hom_{\Z}(0,n)}R\\
  &=
    \begin{cases}
      R\cdot \{0\to n\}&:n\geq 0\\
      0&:n<0
    \end{cases}
\end{align*}
and under this identification, $\{0\to n\}$ corresponds to~$\beta^{n}$.

In the remainder of this section we outline two proofs of the second part of \cref{main-thm}, and deduce the classification of tt-ideals in $\fper(R)$ in \cref{classification-tt-ideals}. The subsequent sections will provide the missing details.

\begin{proof}[First proof of \cref{main-thm}.(2).]  
  It is proven in [\citealp[5.1]{ambrogio-stevenson:graded-ring} ($R$ noetherian); \cite[4.7]{ambrogio-stevenson:tt-comparison-2-rings} ($R$ general)] that the comparison map
  \begin{equation*}
    \rho^{\bullet}:\spc(\D{R[\beta]}^{\mathrm{perf}}_{\scriptscriptstyle\mathrm{gr}})\to\spch(R[\beta])
  \end{equation*}
  is a homeomorphism, where $\D{R[\beta]}_{\scriptscriptstyle\mathrm{gr}}^{\mathrm{perf}}$ denotes the thick subcategory of compact objects in $\D{\GMd(R[\beta])}$. It then follows from \cref{fmod-seq-equivalence} and \cref{fmod-seq-equivalence-tensor} (as well as the identification $\seq{R}\cong\GMd(R[\beta])$ discussed in \cref{sec:fmod}) that the same is true for $\rho^{\bullet}_{R}:\spc(\fper(R))\to\spch(R[\beta])$. By~\cite[6.11]{balmer:sss}, the morphism of locally ringed spaces $\rho^{\bullet}_{R}$ is then automatically an isomorphism.
\end{proof}

For the second proof of \cref{main-thm}.(2) we proceed as follows. By~\cite[6.11]{balmer:sss}, it suffices to show that
\begin{equation*}
  \rho^{\bullet}_{R}:\spc(\fper(R))\to\spch({\mathcal R}^{\bullet}_{R})
\end{equation*}
is a homeomorphism on the underlying topological spaces.

Consider the invertible object $R\in\per(R)$ and the associated graded central ring $R[t,t^{-1}]$ where $t=\id:R\to R$ has degree~1.  The morphisms of graded $R$-algebras induced by $\gr$ and $\pi$ respectively are given by
\begin{align}\label{gr-pi-ring-morphisms}
  \xymatrix@R=0em{R[\beta]\ar[r]^-{\gr}& R[t,t^{-1}]&&R[\beta]\ar[r]^-{\pi}&R[t,t^{-1}]\\
  \beta\ar@{|->}[r]& 0&&\beta\ar@{|->}[r]& t}
\end{align}

Recall (\cref{sec:fmod-derived}) the existence of a section $\sigma_{0}$ to $\gr$ and~$\pi$. We therefore obtain for $\xi\in\{\gr,\pi\}$ commutative diagrams of topological spaces and continuous maps
\begin{equation*}
  \xymatrix{\spc(\per(R))\ar[r]^{\spc(\xi)}\ar[d]_{\rho^{\bullet}}\ar@/_4pc/[dd]_{\rho}&\spc(\fper(R))\ar[r]^{\spc(\sigma_{0})}\ar[d]_{\rho^{\bullet}_{R}}&\spc(\per(R))\ar[d]^{\rho^{\bullet}}\ar@/^4pc/[dd]^{\rho}\\
    \spch(R[t,t^{-1}])\ar[r]^{\spch(\xi)}\ar[d]_{\sim}&\spch(R[\beta])\ar[r]^{\spch(\sigma_{0})}\ar[d]&\spch(R[t,t^{-1}])\ar[d]^{\sim}\\
    \spc(R)\ar[r]_{=}^{\spc(\xi)}&\spc(R)\ar[r]_{=}^{\spc(\sigma_{0})}&\spc(R)}
\end{equation*}
where the outer vertical maps are all homeomorphisms~\cite[8.1]{balmer:sss}, and the composition of the horizontal morphisms in each row is the identity. It follows immediately that both $\spc(\gr)$ and $\spc(\pi)$ are homeomorphisms onto their respective images which are disjoint by \cref{gr-pi-ring-morphisms}. More precisely, we have
\begin{align}\label{image-gr-pi}
 \img(\spc(\gr))&\subseteq (\rho_{R}^{\bullet})^{-1}(Z(\beta))=\supp(\cone(\beta)),\\\img(\spc(\pi))&\subseteq (\rho_{R}^{\bullet})^{-1}(U(\beta))=U(\cone(\beta)).\notag
\end{align}
It now remains to prove two things:
\begin{itemize}
\item $\spc(\gr)$ and $\spc(\pi)$ are jointly surjective.
\item Specializations lift along $\rho^{\bullet}_{R}$.
\end{itemize}
Indeed, since $\rho^{\bullet}_{R}$ is a spectral map between spectral spaces~\cite[5.7]{balmer:sss}, it being a homeomorphism is equivalent to it being bijective and lifting specializations~\cite[8.16]{hochster:spectral}.

The first bullet point is the subject of the subsequent sections. Let us assume it for now and establish the second bullet point. In particular, we now assume that the inclusions in \cref{image-gr-pi} are equalities. Let $\rho_{R}^{\bullet}(\mathfrak{P})\leadsto\rho_{R}^{\bullet}(\mathfrak{Q})$ be a specialization in $\spch(R[\beta])$, \ie~$\rho_{R}^{\bullet}(\mathfrak{P})\subset\rho_{R}^{\bullet}(\mathfrak{Q})$. If $\beta\notin\rho_{R}^{\bullet}(\mathfrak{Q})$ then both primes lie in the image of $\spc(\pi)$ and we already know that~$\mathfrak{P}\leadsto\mathfrak{Q}$. Similarly, if $\beta\in\rho_{R}^{\bullet}(\mathfrak{P})$ then both primes lie in the image of $\spc(\gr)$ and we deduce again that $\mathfrak{P}\leadsto\mathfrak{Q}$. So we may assume~$\beta\in\rho_{R}^{\bullet}(\mathfrak{Q})\backslash\rho_{R}^{\bullet}(\mathfrak{P})$. Define $\mathfrak{r}=\rho_{R}^{\bullet}(\mathfrak{Q})\cap R\in\spc(R)$ and notice that
\begin{equation*}
  \rho_{R}^{\bullet}(\mathfrak{P})\subset\mathfrak{r}[\beta]\subset\mathfrak{r}+\langle\beta\rangle=\rho_{R}^{\bullet}(\mathfrak{Q}).
\end{equation*}
Consequently, the preimage of $\mathfrak{r}[\beta]$ under $\rho^{\bullet}_{R}$ is the prime
\begin{equation*}
  \mathfrak{R}=\ker\left(\fper(R)\xrightarrow{\pi}\per(R)\to \per(R/\mathfrak{r})\right)
\end{equation*}
which contains the prime
\begin{equation*}
  \mathfrak{Q}=\ker\left(\fper(R)\xrightarrow{\gr}\per(R)\to \per(R/\mathfrak{r})\right).
\end{equation*}
We now obtain specialization relations
\begin{equation*}
  \mathfrak{P}\leadsto\mathfrak{R}\leadsto\mathfrak{Q}
\end{equation*}
and the proof is complete.

As a consequence of \cref{main-thm} we will classify the tt-ideals in~$\fper(R)$.
\begin{lem}\label{fper-thomason-subsets}
  The following two maps set up an order preserving bijection
  \begin{align*}
    \left\{\Pi\subset\Gamma \mid \Pi, \Gamma\subset\spc(R)\text{ Thomason subsets}
    \right\}&\longleftrightarrow
              \left\{
              \text{Thomason subsets of }\spc(\fper(R))
              \right\} \\
     (\Pi\subset \Gamma)&\longmapsto\spc(\pi)(\Pi) \sqcup\spc(\gr)(\Gamma)\\
    (\spc(\pi)^{-1}(Y)\subset\spc(\gr)^{-1}(Y))&\longmapsfrom Y
  \end{align*}
\end{lem}
Here, the order relation on the left is given by $(\Pi\subset \Gamma)\leq (\Pi'\subset \Gamma')$ if $\Pi\subset \Pi'$ and~$\Gamma\subset \Gamma'$.
\begin{proof}
  To ease the notation, let us denote in this proof by $p:S\to T$ (respectively, $g:S\to T$) the map $\spc(\pi):\spc(R)\to\spc(\fper(R))$ (respectively~$\spc(\gr)$). It might be helpful to keep the following picture in mind.
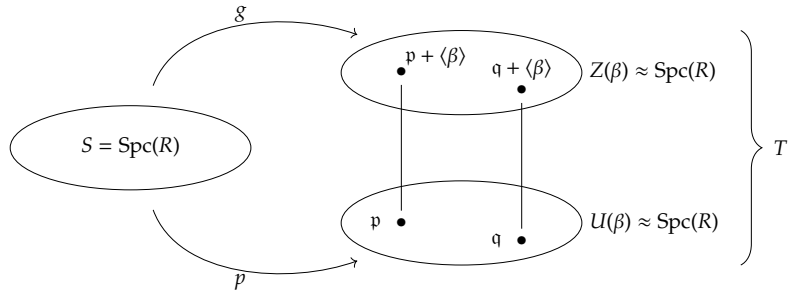
\begin{center}
  \scalebox{0.8}{
    \begin{tikzpicture}
      \node [draw, ellipse, minimum width=4cm, minimum height=1.4cm]
      (domain) at (-5.5,1.25) {$S=\mathrm{Spc}(R)$}; \node [draw,
      ellipse, minimum width=4cm, minimum height=1.4cm] (target-p) at
      (0,0) {}; \node[right] at (target-p.0)
      {$U(\beta)\approx\mathrm{Spc}(R)$}; \node [draw, ellipse,
      minimum width=4cm, minimum height=1.4cm] (target-g) at (0,2.5)
      {}; \node[right] at (target-g.0)
      {$Z(\beta)\approx\mathrm{Spc}(R)$};

      \draw[->,shorten >=10pt,shorten <=10pt] (domain) to
      [out=70,in=160] node[midway,above]{$g$} (target-g);
      \draw[->,shorten >=10pt,shorten <=10pt] (domain) to
      [out=-70,in=-160] node[midway,below]{$p$} (target-p);

      \node[label=left:$\mathfrak{p}$] (p) at (-1,0) {$\bullet$};
      \node[label={[label
        distance=-.35cm]35:$\mathfrak{p}+\langle\beta\rangle$}] (pb)
      at (-1,2.5) {$\bullet$}; \draw[-] (p.north) to (pb.south);
      \node[label=left:$\mathfrak{q}$] (q) at (1,-.3) {$\bullet$};
      \node[label={[label
        distance=-.15cm]90:$\mathfrak{q}+\langle\beta\rangle$}] (qb)
      at (1,2.2) {$\bullet$}; \draw[-] (q.north) to (qb.south); \draw
      [decorate,decoration={brace,amplitude=10pt,mirror,raise=4pt},yshift=0pt]
      (4.5,-0.7) -- (4.5,3.2) node [black,midway,xshift=0.8cm] {$T$};
    \end{tikzpicture}
  }
\end{center}
Thus $g$ and $p$ are spectral maps between spectral spaces, homeomorphisms onto disjoint images which jointly make up all of~$T$. Moreover, the image of $p$ is open, and there is a common retraction $r:T\to S$ to both $g$ and~$p$.

First, the preimages of a Thomason subset $Y\subset T$ under the spectral maps $g$ and $p$ are Thomason. Moreover, every Thomason subset is closed under specializations from which one deduces $p^{-1}(Y)\subset g^{-1}(Y)$. This shows that the map from right to left is well-defined.

Next, given $\Pi\subset \Gamma\subset\spc(R)$ two Thomason subsets we claim that $p(\Pi)\sqcup g(\Gamma)$ is Thomason as well. By assumption, we may write $\Pi=\cup_{i}\Pi_{i}$, $\Gamma=\cup_{j}\Gamma_{j}$ with $\Pi_{i}^{c},\Gamma_{j}^{c}$ quasi-compact open subsets, and hence also
\begin{align*}
  \Pi=\Pi\cap \Gamma=(\cup_{i} \Pi_{i})\cap (\cup_{j} \Gamma_{j})=\cup_{i,j}(\Pi_{i}\cap \Gamma_{j})
\end{align*}
with $(\Pi_{i}\cap \Gamma_{j})^{c}=\Pi_{i}^{c}\cup \Gamma_{j}^{c}$ quasi-compact open.
Then
\begin{align*}
  p(\Pi)\sqcup g(\Gamma)=(\cup_{i,j} p(\Pi_{i}\cap \Gamma_{j}))\sqcup (\cup_{j} g(\Gamma_{j}))=\cup_{i,j}
                                        \left(
                                        p(\Pi_{i}\cap \Gamma_{j})\sqcup g(\Gamma_{j})
                                        \right)
\end{align*}
and we reduce to the case where $\Pi^{c}$ and $\Gamma^{c}$ are quasi-compact open. But in that case,
\begin{align*}
  \left(
  p(\Pi)\sqcup g(\Gamma)
  \right)^{c}=
  \left(
  p(\Gamma^{c})\sqcup g(\Gamma^{c})
  \right)\cup p(\Pi^{c})=r^{-1}(\Gamma^{c})\cup p(\Pi^{c}).
\end{align*}
Again, $r$ is a spectral map and hence the first set is quasi-compact open. The second one is quasi-compact by assumption, and also open since $p$ is a homeomorphism onto an open subset. This shows that the map from left to right is also well-defined.

It is obvious that the two maps are order preserving and inverses to each other.
\end{proof}

To state the classification result more succinctly, let us make the following definition.
\begin{dfn}
  Let $a\in\fper(R)$. For $\xi\in\{\pi,\gr\}$ set
  \begin{equation*}
    \supp_{\xi}(a):=\{\mathfrak{p}\in\spc(R)\mid \xi(a\otimes \kappa(\mathfrak{p}))\neq 0\in\per(\kappa(\mathfrak{p}))\}.
  \end{equation*}
  We extend this definition to arbitrary subsets ${\mathcal J}\subset \fper(R)$ by
  \begin{equation*}
    \supp_{\xi}({\mathcal J}):=\bigcup_{a\in{\mathcal J}}\supp_{\xi}(a).
  \end{equation*}
\end{dfn}

\begin{lem}\label{support-basic}
  Let~$a\in\fper(R)$. Then:
  \begin{enumerate}
  \item $\supp_{\gr}(a)=\{\mathfrak{p}\in\spc(R)\mid a\otimes \kappa(\mathfrak{p})\neq 0\in\fper(\kappa(\mathfrak{p}))\}$.
  \item $\supp_{\pi}(a)\subset\supp_{\gr}(a)$.
  \item $\supp_{\xi}(a)=\supp(\xi(a))$.
  \end{enumerate}
\end{lem}
\begin{proof}
  \begin{enumerate}
  \item The functor $\gr:\fper(\kappa(\mathfrak{p}))\to\per(\kappa(\mathfrak{p}))$ is conservative by \cref{gr-derived}, thus the claim.
  \item This follows immediately from the first part.
  \item We have
    \begin{align*}
      \supp_{\xi}(a)&=\{\mathfrak{p}\in\spc(R)\mid \xi(a\otimes \kappa(\mathfrak{p}))\neq 0\in\per(\kappa(\mathfrak{p}))\}\\
      &=\{\mathfrak{p}\in\spc(R)\mid \xi(a)\otimes \kappa(\mathfrak{p})\neq 0\in\per(\kappa(\mathfrak{p}))\}\\
      &=\supp(\xi(a)).
    \end{align*}
  \end{enumerate}
\end{proof}
The relation to the usual support can be expressed in two (equivalent) ways.
\begin{lem}\label{support-decomposition}
  Let~$a\in\fper(R)$. Then
  \begin{enumerate}
  \item $\supp(a)=
    \spc(\pi)(\supp_{\pi}(a))\sqcup\spc(\gr)(\supp_{\gr}(a))$.
  \item Under the bijection of \cref{fper-thomason-subsets}, $\supp(a)$ corresponds to the pair~$\supp_{\pi}(a)\subset\supp_{\gr}(a)$.
  \end{enumerate}
\end{lem}
\begin{proof}
  Both statements follow from
  \begin{equation*}
    \spc(\xi)^{-1}(\supp(a))=\supp(\xi(a))=\supp_{\xi}(a),
  \end{equation*}
  the last equality being true by \cref{support-basic}.
\end{proof}

\begin{lem}\label{tt-ideal-K}
  Let $Y\subset\spc(\fper(R))$ be a Thomason subset, corresponding to $\Pi\subset \Gamma$ under the bijection in \cref{fper-thomason-subsets}. For $a\in\fper(R)$ the following are equivalent:
  \begin{enumerate}
  \item $\supp(a)\subset Y$;
  \item $\supp_{\pi}(a)\subset \Pi$ and~$\supp_{\gr}(a)\subset \Gamma$.
  \end{enumerate}
\end{lem}
\begin{proof}
  This follows immediately from the way $\Pi\subset \Gamma$ is associated to $Y$, together with \cref{support-decomposition}.
\end{proof}

\begin{cor}\label{classification-tt-ideals}
    There is an inclusion preserving bijection
  \begin{align*}
    \left\{\Pi\subset \Gamma\mid \Pi, \Gamma\subset\spc(R)\text{ Thomason subsets}
    \right\}&\longleftrightarrow
              \left\{
              \text{tt-ideals in }\fper(R)
              \right\} \\
     (\Pi\subset \Gamma)&\longmapsto \{a\mid \supp_{\pi}(a)\subset \Pi,\supp_{\gr}(a)\subset \Gamma\}\\
    \left(
    \supp_{\pi}({\mathcal J})\subset\supp_{\gr}({\mathcal J})
    \right)&\longmapsfrom {\mathcal J}
  \end{align*}
\end{cor}
\begin{proof}
A bijection between Thomason subsets of $\spc(\fper(R))$ and tt-ideals in $\fper(R)$ is described in~\cite[14]{balmer:icm}. Explicitly, it is given by $Y\mapsto \{a\mid\supp(a)\subset Y\}$ and~$\supp({\mathcal J})\mapsfrom{\mathcal J}$. The Corollary follows by composing this bijection with the one of \cref{fper-thomason-subsets}, using \cref{support-decomposition} and \cref{tt-ideal-K}.
\end{proof}

\section{Central localization}
\label{sec:localization}
In this section we study several localizations of $\fper(R)$ which will allow us to catch primes (points for the tt-spectrum). In order to accommodate the different localizations we are interested in, we want to work in the following setting. Let $\A$ be a tensor category with central ring $R=\hom_{\A}(\one,\one)$, and fix an invertible object~$u\in\A$. Most of the discussion in~\cite[section 3]{balmer:sss} regarding graded homomorphisms and central localization carries over to our setting. Let us recall what we will need from~\loccit

The graded central ring of $\A$ with respect to $u$ is~$\R^{\bullet}=\hom_{\A}(\one,u^{\otimes\bullet})$. This is a $\Z$-graded $\varepsilon$-commutative ring, where $\varepsilon\in R$ is the central switch for $u$, \ie the switch $u\otimes u\xrightarrow{\sim}u\otimes u$ is given by multiplication by~$\varepsilon$. For any objects $a,b\in\A$, the $\Z$-graded abelian group $\hom_{\A}^{\bullet}(a,b)=\hom_{\A}(a,b\otimes u^{\otimes \bullet})$ has the structure of a graded $\R^{\bullet}$-module in a natural way.

Let $S\subset \R^{\bullet}$ be a multiplicative subset of central homogeneous elements. The central localization $S^{-1}\A$ of $\A$ with respect to $S$ is obtained as follows: it has the same objects as $\A$, and for $a,b\in\A$,
\begin{equation*}
  \hom_{S^{-1}\A}(a,b)=
  \left(
    S^{-1}\hom_{\A}^{\bullet}(a,b)
  \right)^{0},
  \end{equation*}
  the degree 0 elements in the graded localization.

We now prove that this is in fact a categorical localization.
\begin{pro}\label{central-categorical-localization}
  The canonical functor ${\mathcal Q}:\A\to S^{-1}\A$ is the localization with respect to
  \begin{equation*}
    \Sigma=\{a\xrightarrow{s} a\otimes u^{\otimes n}\mid a\in\A, s\in S, |s|=n\}.
  \end{equation*}
  Moreover, $S^{-1}\A$ has a canonical structure of a tensor category, and ${\mathcal Q}$ is a tensor functor.
\end{pro}
\begin{proof}
    Denote by ${\mathcal Q}'$ the localization functor~$\A\to\Sigma^{-1}\A$. It is clear by construction that every morphism in $\Sigma$ is inverted in $S^{-1}\A$ thus ${\mathcal Q}$ factors through ${\mathcal Q}'$, say via the functor~$F:\Sigma^{-1}\A\to S^{-1}\A$. The functor $F$ is clearly essentially surjective. And fully faithfulness follows readily from the fact, easy to verify, that $\Sigma$ admits a calculus of left (and right) fractions~\cite[2.2]{gabriel-zisman:localization}.

    The fact that $\Sigma^{-1}\A$ is an additive category and ${\mathcal Q}'$ an additive functor is~\cite[3.3]{gabriel-zisman:localization}, and the analogous statement about the monoidal structure is proven in~\cite{day:monoidal-localization}. The monoidal product in $\Sigma^{-1}\A$ is automatically additive in each variable.
\end{proof}

Consider the homotopy category $\K{\A}^{b}$ of~$\A$. This is a tt-category (large if $\A$ is) with the same graded central ring $\R^{\bullet}$ (with respect to $u$ considered in degree 0).
\begin{lem}\label{localization-homotopy}
  There is a canonical equivalence of tt-categories $S^{-1}\K{\A}^{b}\simeq \K{S^{-1}\A}^{b}$, and both are equal to the Verdier localization of $\K{\A}^{b}$ with kernel~$\langle\cone(s)\mid s\in S\rangle$.
\end{lem}
\begin{proof}
  The first statement can be shown in two steps. First, consider the category of chain complexes $\C{\A}^{b}$ and the canonical functor~$\C{\A}^{b}\to \C{S^{-1}\A}^{b}$. By \cref{central-categorical-localization}, it factors through $S^{-1}\C{\A}^{b}\to \C{S^{-1}\A}^{b}$; fully faithfulness and essential surjectivity of this functor is an easy exercise using the explicit nature of the central localization. (The point is that for bounded complexes there are always only finitely many morphisms involved thus the possibility of finding a ``common denominator''.)

  Next, since $\C{-}^{b}\to \K{-}^{b}$ is a categorical localization (with respect to chain homotopy equivalences), \cref{central-categorical-localization} easily implies the claim.

  Compatibility with the tt-structure is also straightforward. The second statement follows from~\cite[3.6]{balmer:sss}.
\end{proof}

We want to draw two consequences from this discussion. For the first one, denote by $\md(R)$ the tensor category of rigid objects in $\Md(R)$, \ie the category of finitely generated projective $R$-modules. We let $\A=\fmd(R)$ and as the invertible object $u$ we choose $R(1)$ so that $\R^{\bullet}=R[\beta]$.
\begin{lem}\label{localization-beta-fmod}
  The functor $\pi:\fmd(R)\to\md(R)$ is the central localization at the multiplicative set~$\{\beta^{n}\mid n\geq 0\}\subset R[\beta]$.
\end{lem}
\begin{proof}
  Consider the set of arrows $\Sigma=\{\beta^{n}:a\to a(n)\mid a\in \fmd(R), n\geq 0\}$. By \cref{central-categorical-localization}, the central localization in the statement of the Lemma is the localization at~$\Sigma$. We have $\Sigma^{-1}\fmd(R)(a,b)=\varinjlim_{n}\hom_{\fmd(R)}(a(-n),b)$. At each level $n$, this maps injectively into $\hom_{\md(R)}(\pi a,\pi b)$, and the transition maps $f\mapsto f\circ\beta$ are injective as well since $\beta$ is an epimorphism, hence the induced map
  \begin{equation*}
    \varinjlim_{n}\hom_{\fmd(R)}(a(-n),b)\to\hom_{\md(R)}(\pi{a},\pi{b})
  \end{equation*}
  is injective. For surjectivity, we may assume $a,b\in \fmd(R)$ are of ``weight in $[m,n]$'', \ie $m\leq n$ and $\gr_{i}(a)=\gr_{i}(b)=0$ for all~$i\notin [m,n]$. In that case $f:\pi a\to\pi b$ comes from a map~$f:a(m-n)\to b$.

  We have proved that $\pi:\Sigma^{-1}\fmd(R)\to\md(R)$ is fully faithful. Essential surjectivity is clear. 
\end{proof}

\begin{cor}\label{localization-beta-derived-fmod}
  The functor $\pi:\fper(R)\to\per(R)$ is the Verdier localization at the morphisms $\beta:A\to A(1)$, every~$A\in\fper(R)$. In particular, $\ker(\pi)=\langle\cone(\beta)\rangle$.
\end{cor}
\begin{proof}
  Let $S=\{\beta^{n}\}\subset R[\beta]$. We know from \cref{localization-beta-fmod} that $S^{-1}\fmd(R)=\md(R)$ hence also $S^{-1}\fper(R)=\per(R)$, by \cref{localization-homotopy}, and this is the Verdier localization with kernel $\langle\cone(\beta^{n})\mid n\geq 0\rangle$. The latter tt-ideal is equal to $\langle\cone(\beta)\rangle$ by~\cite[2.16]{balmer:sss} and we conclude.
\end{proof}

Still in the same context let $\mathfrak{p}\subset R$ be a prime ideal. Denote by $q:R\to R_{\mathfrak{p}}$ the canonical localization morphism, and set~$S=R\backslash \mathfrak{p}$.
\begin{lem}\label{localization-fmod}
  The morphism $q$ induces an equivalence of tensor categories
  \begin{equation*}
    S^{-1}\fmd(R)\simeq\fmd(R_{\mathfrak{p}}).
  \end{equation*}
\end{lem}
\begin{proof}
  The functor $S^{-1}\fmd(R)\to\fmd(R_{\mathfrak{p}})$ is given by~$\otimes_{R}R_{\mathfrak{p}}$. This is clearly a tensor functor. Since $R_{\mathfrak{p}}$ is local every finitely generated projective $R_{\mathfrak{p}}$-module is free thus $\otimes_{R}R_{\mathfrak{p}}$ is essentially surjective. For full faithfulness notice that $\otimes_{R}R_{\mathfrak{p}}$ is additive and one therefore reduces to check this for twists of $R$:
  \begin{align*}
S^{-1}\hom_{\fmd(R)}(R(m),R(n))&=
                                 \begin{cases}
                                   S^{-1}R&:n\geq m\\
                                   0&:n<m
                                 \end{cases}\\
    &=
                                 \begin{cases}
                                   R_{\mathfrak{p}}&:n\geq m\\
                                   0&:n<m
                                 \end{cases}\\
&=\hom_{\fmd(R_{\mathfrak{p}})}(R_{\mathfrak{p}}(m),R_{\mathfrak{p}}(n)).
  \end{align*}
\end{proof}

\begin{cor}\label{localization-fmod-derived}
  The square of topological spaces
  \begin{equation*}
        \xymatrix{\spc(\fper(R))\ar[d]_{\rho_{R}}&\spc(\fper(R_{\mathfrak{p}}))\ar[l]_{\spc(q)}\ar[d]^{\rho_{R_{\mathfrak{p}}}}\\
      \spc(R)&\spc(R_{\mathfrak{p}})\ar[l]^{\spc(q)}}
  \end{equation*}
  is cartesian.
\end{cor}
\begin{proof}
 By \cref{localization-homotopy} and \cref{localization-fmod}, we know that $\fper(R_{\mathfrak{p}})$ is the Verdier localization of $\fper(R)$ with kernel $\langle\cone(s)\mid s\in S\rangle$. The claim now follows from~\cite[5.6]{balmer:sss}.
\end{proof}

\begin{rmk}
  \cref{localization-fmod} is false for general multiplicative subsets $S\subset R$, even without taking into account filtrations. The proof shows that the functor $S^{-1}\fmd(R)\to\fmd(S^{-1}R)$ is always fully faithful but it may fail to be essentially surjective. The correct statement would therefore be that $
  \left(
    S^{-1}\fmd(R)
\right)^{\natural}\simeq\fmd(S^{-1}R)$, where $(-)^{\natural}$ denotes the idempotent completion. We then deduce
\begin{align*}
  \K{\fmd(S^{-1}R)}^{b}&\simeq \K{\left(S^{-1}\fmd(R)\right)^{\natural}}^{b}\\
  &\simeq\left(\K{S^{-1}\fmd(R)}^{b}\right)^{\natural}&&\text{\cite[2.8]{Balmer-Schlichting}}\\
  &\simeq\left(S^{-1}\K{\fmd(R)}^{b}\right)^{\natural}&&\text{\cref{localization-homotopy}}
\end{align*}
and since the tt-spectrum is invariant under idempotent completion, we obtain a cartesian square as in \cref{localization-fmod-derived} for arbitrary multiplicative subsets~$S\subset R$.
\end{rmk}
\section{Reduction steps}
\label{sec:reduction}
Let $R$ be a noetherian ring. Recall from \cref{sec:main-thm} that we would like to prove that the tt-functors $\pi,\gr:\fper(R)\to\per(R)$ induce jointly surjective maps
\begin{equation*}
  \spc(\pi),\spc(\gr):\spc(\per(R))\to\spc(\fper(R)).
\end{equation*}
In this section, we will explain how to reduce this statement to $R$ a field. The latter case will be proved in \cref{sec:field}, and the case of arbitrary (\ie not necessarily noetherian) rings will be addressed in \cref{sec:continuity}.

\begin{pro}\label{reduction-nilpotent}
  If $r\in R$ is nilpotent then the canonical map 
  \begin{equation*}
    \spc(\fper(R/r))\to\spc(\fper(R))
  \end{equation*}
  is surjective.
\end{pro}
\begin{proof}
  Let $F=\otimes_{R}R/r:\fper(R)\to\fper(R/r)$. We will use the criterion in~\cite[1.3]{balmer:surjectivity} to establish surjectivity of $\spc(F)$, \ie we want to prove that $F$ detects $\otimes$-nilpotent morphisms. Let $f:A\to B\in\fper(R)$ such that~$\overline{f}:=F(f)=0$. Equivalently, we may consider $f$ as a morphism $f':R(0)\to A^{\vee}\otimes B$, where $A^{\vee}$ denotes the dual of $A$. Then $\overline{f'}=0$ and if $(f')^{\otimes m}=0$ then also $f^{\otimes m}=0$, in other words we reduce to $A=R(0)$.

  The morphism $f$ in $\fper(R)$ is then determined by a map $f^{0}:R(0)\to B^{0}$ such that $\delta^{0}f^{0}=0$, and $\overline{f}=0$ means that there is a map $\overline{h}:R/r(0)\to B^{-1}/r$ such that $\overline{f^{0}}=\overline{\delta^{-1}}\overline{h}$. Choose a lift $h:R(0)\to B^{-1}$ of $\overline{h}$ to $\fmd(R)$. There exists $g:R(0)\to B^{0}$ such that $f^{0}-gr=\delta^{-1}h$. The composite $gr$ determines a chain morphism, and we may assume that $f^{0}$ is of the form $gr$ for some $g:R(0)\to B^{0}$. (The map $g$ itself does not necessarily determine a chain morphism.)

Let $m\geq 1$ such that $r^{\circ m}=0$. Then $f^{\otimes m}:R(0)\to B^{\otimes m}$ is described by the morphism
\begin{align*}
  R(0)\xrightarrow{(gr)^{\otimes m}} (B^{0})^{\otimes m}\hookrightarrow (B^{\otimes m})^{0}\\
\intertext{which factors as}\\
R(0)\xrightarrow{r^{\circ m}=0}R(0)\xrightarrow{g^{\otimes m}} (B^{0})^{\otimes m}\hookrightarrow (B^{\otimes m})^{0}.
\end{align*}
We conclude that $f$ is $\otimes$-nilpotent as required.
\end{proof}

\begin{pro}\label{reduction-nzd}
  Let $r\in R$ be a non-zerodivisor. The image of the canonical map
  \begin{equation*}
    \spc(\fper(R/r))\to\spc(\fper(R))
  \end{equation*}
  is precisely the support of $\cone(r)$.
\end{pro}
\begin{proof}
  Let $F=\otimes_{R}R/r:\fper(R)\to\fper(R/r)$. The fact that $r$ is a non-zerodivisor means that $R/r(0)$ is an object in $\fper(R)$ hence $F$ admits a right adjoint $G:\fper(R/r)\to\fper(R)$ (which is simply the forgetful functor). We may therefore invoke~\cite[1.7]{balmer:surjectivity}: the image of $\spc(F)$ is the support of~$G(R/r(0))=\cone(r)$.
\end{proof}

We can now put these pieces together. Notice that we have, for any ring morphism $R\to R'$ and $\xi\in\{\pi,\gr\}$, commutative squares
\begin{equation}\label{basechange-prime-generators}
  \xymatrix@C=5em{\spc(\fper(R))&\spc(\fper(R'))\ar[l]_{\spc(\otimes_{R}R')}\\
    \spc(\per(R))\ar[u]^{\spc(\xi)}&\spc(\per(R'))\ar[l]^{\spc(\otimes_{R}R')}\ar[u]_{\spc(\xi)}.}
\end{equation}
Let $\mathfrak{P}\in\spc(\fper(R))$ be a prime and set~$\mathfrak{p}=\rho_{R}(\mathfrak{P})\in\spc(R)$. From \cref{localization-fmod-derived} we know that $\mathfrak{P}$ lies in the subspace~$\spc(\fper(R_{\mathfrak{p}}))$. Using \cref{basechange-prime-generators} we therefore reduce to a local ring $(R,\mathfrak{p})$ (still assuming $\mathfrak{p}=\rho_{R}(\mathfrak{P})$). We now do induction on the dimension $d$ of~$R$. In each case, repeated application of \cref{reduction-nilpotent} in conjunction with \cref{basechange-prime-generators} allows us to assume $R$ reduced. If $d=0$, $R$ is necessarily a field and this case will be dealt with in \cref{field-spectrum}. If $d>0$ there exists a non-zerodivisor~$r\in \mathfrak{p}$. \cref{reduction-nzd} in conjunction with \cref{basechange-prime-generators} reduce us to $R/r$ but this ring has dimension $d-1$ and we conclude by induction.

\section{The case of a field}
\label{sec:field}
In this section we will prove \cref{main-thm} in the case of $R=k$ a field. This will follow easily from a more precise description of~$\fper(k)$.

We begin with a result describing the structure of any morphism in~$\fmd(k)$. For this, let us agree to call a quasi-abelian category  \emph{semisimple} if every short strictly exact sequence splits. Equivalently, a quasi-abelian category is semisimple if every object is projective.

\begin{lem}
  The category $\fmd(k)$ is semisimple quasi-abelian.
\end{lem}
\begin{proof}
  Notice that $\fmd(k)\subset \FMd(k)$ is simply the full subcategory of separated filtered vector spaces whose underlying vector space is finite dimensional. This is an additive subcategory and the set of objects is closed under kernels and cokernels in $\FMd(k)$. We deduce that it is a quasi-abelian subcategory.

 Since every object in $\fmd(k)$ is projective (\cref{fmod-finite-projectives}), semisimplicity follows.
\end{proof}

\begin{lem}\label{field-morphism-structure}
  Let $f:a\to b$ be a morphism in a semisimple quasiabelian category. Then $f$ is the composition
  \begin{equation}
    \label{field-morphism-composition}
    f=f_{m} \circ f_{em}\circ f_{e},
  \end{equation}
  where
  \begin{itemize}
  \item $f_{e}$ is the projection onto a direct summand (in particular a strict epimorphism),
  \item $f_{em}$ is an epimonomorphism,
  \item $f_{m}$ is the inclusion of a direct summand (in particular a strict monomorphism).
  \end{itemize}
\end{lem}
\begin{proof}
  As in every quasi-abelian category, $f$ factors as
  \begin{equation*}
    a\xrightarrow{f_{e}}\coimg(f)\xrightarrow{f_{em}}\img(f)\xrightarrow{f_{m}} b,
  \end{equation*}
  where $f_{e}$ is a strict epimorphism, $f_{em}$ is an epimonomorphism, and $f_{m}$ is a strict monomorphism. The Lemma now follows from the definition of semisimplicity.
\end{proof}

\begin{rmk}\label{field-morphism-characterization}
  \cref{field-morphism-structure} allows to characterize certain properties of morphisms $f:a\to b$ in a particularly simple way:
  \begin{enumerate}
  \item $f$ is a monomorphism if and only if $f_{e}$ is an isomorphism.
  \item $f$ is an epimorphism if and only if $f_{m}$ is an isomorphism.
  \item $f$ is strict if and only if $f_{em}$ is an isomorphism.
  \end{enumerate}
\end{rmk}

Fix a semisimple quasi-abelian category~${\mathcal A}$. Its bounded derived category $\D{\mathcal A}^{b}$ admits a bounded t-structure whose heart $\D{\mathcal A}^{\heartsuit}$ is the subcategory of objects represented by complexes
\begin{equation}
  0\to a\xrightarrow{f}b\to 0,\label{left-heart-complex}
\end{equation}
where $b$ sits in degree 0 and $f$ is a monomorphism in ${\mathcal A}$.\footnote{This is~\cite[1.2.18, 1.2.21]{Schneiders:quasi-abelian}. The reader who is puzzled by the asymmetry of this statement should rest assured that there is a dual t-structure for which the objects in the heart are represented by \emph{epi}morphisms~\cite[1.2.23]{Schneiders:quasi-abelian}. Also, the existence of the t-structures does not require $\A$ to be semisimple.}

\begin{lem}\label{semisimple-hereditary}
  The t-structure on $\D{\mathcal A}^{b}$ is strongly hereditary, \ie for any $A,B\in \D{\mathcal A}^{\heartsuit}$ and $i\geq 2$, we have~$\hom_{\D{\mathcal A}^{b}}(A,B[i])=0$.
\end{lem}
\begin{proof}
  This follows from the fact that $A$ and $B$ are represented by complexes of the form (\ref{left-heart-complex}), and that homomorphisms can be computed in the homotopy category. Indeed, as every object in ${\mathcal A}$ is projective, the canonical functor $\K{\A}^{b}\to\D{\A}^{b}$ is an equivalence \cite[1.3.22]{Schneiders:quasi-abelian}.
\end{proof}

Assume now in addition that ${\mathcal A}$ is a tensor category and every object is a finite sum of invertibles. Clearly, $\fmd(k)$ satisfies this condition.

\begin{pro}\label{semisimple-derived-objects}
  Every object in $\D{\mathcal A}^{b}$ is of the form
  \begin{equation*}
    \oplus_{i}\cone(g_{i})[i]\bigoplus \oplus_{j} c_{j}[a_{j}],
  \end{equation*}
where the sums are finite, the $c_{j}$ are invertible in ${\mathcal A}$, and the $g_{i}$ are epimonomorphisms in~${\mathcal A}$.
\end{pro}
\begin{proof}
  Let $A\in \D{\mathcal A}^{b}$. By \cref{semisimple-hereditary}, the object $A$ is a finite direct sum of shifts of objects in~$\D{\mathcal A}^{\heartsuit}$. As discussed above, every object in the heart is represented by a complex as in \cref{left-heart-complex}. We then deduce from \cref{field-morphism-characterization} that $f$ is an epimonomorphism $g$ followed by the inclusion of a direct summand, say with direct complement~$c$. Thus
  \begin{equation*}
    \cone(f)=\cone(g)\oplus c.
  \end{equation*}
\end{proof}

We now come to the study of tt-ideals in~$\fper(k)=\D{\fmd(k)}^{b}$. \cref{semisimple-derived-objects} tells us that every prime ideal is generated by cones of epimonomorphisms in~$\fmd(k)$. However, it turns out that all these cones generate the same prime ideal (except if 0, of course).
\begin{pro}\label{field-tt-ideals}
  There is a unique non-trivial, proper tt-ideal in $\fper(k)$ given by
  \begin{equation*}
    \ker(\pi)=\langle\cone(\beta)\rangle.
  \end{equation*}
  In particular, $\langle\cone(\beta)\rangle$ is a prime ideal.
\end{pro}
\begin{proof}
  The equality of the two tt-ideals follows from \cref{localization-beta-derived-fmod}. Since $\pi$ is a tt-functor and $\per(k)$ is local, it is clear that its kernel is a prime ideal.

Let $A$ be a non-zero object in $\fper(k)$ such that~$\langle A\rangle\neq\fper(k)$. We would like to show $\langle A\rangle=\langle\cone(\beta)\rangle$. By \cref{semisimple-derived-objects}, we may assume $A=\cone(g)$ where $g$ is a non-strict epimonomorphism in~$\fmd(k)$. Writing the domain and codomain of $g$ as a sum of invertibles, we may identify $g$ with a square matrix with entries in the polynomial ring~$k[\beta]$. Let $p(\beta)\in k[\beta]$ be the determinant. Since $g$ is not an isomorphism neither is $\gr(g)\in\per(k)$ by \cref{gr-derived}. We deduce that $p(0)=0$, or in other words $p(\beta)=\beta\cdot p'(\beta)$ for some~$p'(\beta)\in k[\beta]$.

Let $\T=\fper(k)/\langle\cone(g)\rangle$ and denote by $\varphi:\fper(k)\to\T$ the localization functor. As $\T$ is a tt-category we can consider the graded (automatically commutative) central ring ${\mathcal R}^{\bullet}_{\T}$ with respect to~$\varphi(k(1))$. Since $\varphi(g)$ is invertible, $\varphi(p)\in{\mathcal R}^{\bullet}_{\T}$ is invertible as well. But then we must have
\begin{equation*}
  \left(
    \varphi(p)^{-1}\cdot\varphi(p')
  \right)\cdot\varphi(\beta)=\varphi(p)^{-1}\cdot\varphi(p)=1
\end{equation*}
so $\varphi(\beta)$ is invertible as well. In other words, $\cone(\beta)\in\ker(\varphi)=\langle{\cone(g)}\rangle$.

Conversely, $\pi(g)$ is an isomorphism since $g$ is an epimonomorphism. In other words,~$\cone(g)\in\langle\cone(\beta)\rangle$.
\end{proof}

\begin{cor}\label{field-spectrum}
  The canonical morphism
  \begin{equation*}
    \rho^{\bullet}_{k}:\spec(\fper(k))\to\spech(k[\beta])
  \end{equation*}
  is an isomorphism of locally ringed spaces. The tt-spectrum $\spc(\fper(k))$ is the topological space
  \begin{equation*}
    \xymatrix{\langle 0\rangle=\ker(\gr)\\
      \langle\cone(\beta)\rangle=\ker(\pi)\ar@{-}[u]}
  \end{equation*}
  where the only non-trivial specialization relation is indicated by the vertical line going upward.
\end{cor}

\section{Continuity of tt-spectra}
\label{sec:continuity}

Our primary goal in this section is to deduce the veracity of \cref{main-thm} from its veracity for noetherian rings. The idea is to write an arbitrary ring as a filtered colimit of noetherian rings, and since this technique of reducing some statement in tt-geometry to the analogous statement about ``more finite'' objects can be useful in other contexts we decided to approach the question in greater generality.

Denote by $\ttCat$ the category of small tt-categories and tt-functors. For the moment we assume that all structure is strict, \eg the tt-functors preserve the tensor product and translation functor on the nose.
\begin{lem}\label{ttcat-filtered-colimits}
  The forgetful functor $\ttCat\to\Cat$ creates filtered colimits.
\end{lem}
\begin{proof}[Proof sketch.]
  The fact that filtered colimits of monoidal categories are created by the forgetful functor is~\cite[C1.1.8]{johnstone:sketches}. Since filtered colimits commute with finite products, the colimit will be an additive category. It is obvious how to endow the filtered colimit with a translation functor and a class of distinguished triangles. The axioms for the triangulated structure all involve only finitely many objects and morphisms each and therefore are easily seen to hold. The same is true for exactness of the monoidal product.

It remains to check universality. But given a cocone on the diagram there is a unique morphism (apriori not respecting the tt-structure) from the filtered colimit. Hence all one needs to know is that it actually does respect the tt-structure. Again, in each case this only involves finitely many objects and morphisms and is easily seen to hold.
\end{proof}

Let us be given a filtered diagram $(\T_{i},f_{ij}:\T_{i}\to \T_{j})_{i\in I}$ in $\ttCat$ and denote by $\T$ its colimit, and by $f_{i}:\T_{i}\to\T$ the canonical functors.
\begin{pro}\label{spc-continuity}
  The induced map 
  \begin{equation*}
    \varphi:=\varprojlim_{i}f_{i}^{-1}:\spc(\T)\to\varprojlim_{i}\spc(\T_{i})
  \end{equation*}
  is a homeomorphism.
\end{pro}
\begin{proof}
  This follows from \cref{spc-continuity-variant}.
\end{proof}

\begin{rmk}\label{filtered-pseudo-colimit}
  In practice, of course, tt-categories and tt-functors are rarely strict, and (filtered) diagrams of such things are rarely strictly functorial. Denote by $\ttCatp$ the 2-category of small tt-categories, tt-functors, and tt-isotransformations without any strictness assumptions.

Given a pseudo-functor $F:I\to\ttCatp$, where $I$ is a small filtered category, we are going to endow its pseudo-colimit 2-$\varinjlim_{I}F$ with the structure of a tt-category. For this, choose a strictification of $F$, \ie a strict 2-functor $G:I\to\Cat$ together with a pseudo-natural equivalence $\eta:F\to G$ (as pseudo-functors~$I\to\Cat$). Then use $\eta$ pointwise to endow each category $G(i)$, where $i\in I$, with the structure of a tt-category, and each functor $G(\alpha)$, where $\alpha:i\to j$, with the structure of a tt-functor. In other words, make $\eta$ into a pseudo-natural equivalence of pseudo-functors~$I\to\ttCatp$. Since 2-$\varinjlim F\simeq$ 2-$\varinjlim G$, we may assume without loss of generality that $F$ is a strict 2-functor. But in this case the canonical functor 2-$\varinjlim_{I}F\to\varinjlim_{I}F$ from the pseudo-colimit to the (1-categorical) colimit is an equivalence (here we use the assumption that $I$ is filtered~\cite[see][VI.6.8]{sga4}). Then we can apply \cref{ttcat-filtered-colimits}.\footnote{This is maybe not wholly satisfactory. In analogy to \cref{ttcat-filtered-colimits} one might expect the statement that $\ttCatp\to\Catp$ creates filtered pseudo-colimits. We won't need this at present, and leave it as a question for the interested reader.}

  \cref{spc-continuity} also holds in this non-strict context. Notice first that non-strict tt-functors induce maps on spectra exactly in the same way as strict ones. Moreover, isomorphic (non-strict) tt-functors induce the same map. Therefore the statement of \cref{spc-continuity} makes sense also for pseudo-functors~$I\to\ttCatp$. It is clear that $F\to \text{2-}\varinjlim F$ satisfies the assumptions of \cref{spc-continuity-variant} below, thus a homeomorphism
  \begin{equation*}
    \spc(\text{2-}\varinjlim F)\xrightarrow{\sim}\varprojlim_{i}\spc(F(i)).
  \end{equation*}
\end{rmk}

In order to generalize \cref{spc-continuity} we now abstract the pertinent properties of the relation between the system $(\T_{i},f_{ij})$ and the ``limit'' $\T$.
\begin{dfn}
  Let $\T_{\bullet}:I\to\ttCatp$ be a pseudo-functor and
  $f:\T_{\bullet}\to \T$ a pseudo-natural transformation,~$\T\in\ttCatp$. We say that
  \begin{itemize}
  \item $f$ is \emph{surjective on morphisms} if for each
    morphism $\alpha:a\to b$ in $\T$ there exists $i\in I$, and
    a morphism $\alpha_{i}:a_{i}\to b_{i}$ in $\T_{i}$ such that~$f_{i}(\alpha_{i})\cong \alpha$.
  \item $f$ \emph{detects isomorphisms} if for each
    $a_{i},b_{i} \in \T_{i}$ such that
    $f_{i}(a_{i})\cong f_{i}(b_{i})$ in $\T$ there exists $u:i\to j$
    such that~$\T_{u}(a_{i})\cong \T_{u}(b_{i})$.
  \end{itemize}
\end{dfn}
The condition $f_{i}(\alpha_{i})\cong \alpha$ here means that there are isomorphisms $a\cong f_{i}(a_{i})$ and $b\cong f_{i}(b_{i})$ such that
\begin{equation*}
  \xymatrix{a\ar[r]^{\alpha}\ar@{-}[d]_{\sim}&b\ar@{-}[d]^{\sim}\\
    f_{i}(a_{i})\ar[r]_{f_{i}(\alpha_{i})}&f_{i}(b_{i})}
\end{equation*}
commutes. The transformation $f$ being surjective on morphisms implies in particular that $f$ is ``surjective on objects'' and even ``surjective on triangles'', in an obvious sense. Note also that detecting isomorphisms is equivalent to detecting zero objects.

In the following result a category $I$ is said to be \emph{conjoining} if
\begin{itemize}
\item $I$ is non-empty, and
\item for any $i,j\in I$ there exists $k\in I$ and $i\to k$,~$j\to k$.
\end{itemize}
In contrast to a filtered category, it is not necessary that parallel morphisms can be equalized. (Of course, in applications $I$ will often just be a directed poset.)

\begin{pro}\label{spc-continuity-variant}
  Let $\T_{\bullet}:I\to\ttCatp$ be a pseudo-functor with $I$ conjoining, and $f:\T_{\bullet}\to \T$ a pseudo-natural transformation,~$\T\in\ttCatp$. Assume that $f$ is surjective on morphisms and detects isomorphisms. Then the induced map 
  \begin{equation*}
    \varphi:=\varprojlim_{i}f_{i}^{-1}:\spc(\T)\to\varprojlim_{i}\spc(\T_{i})
  \end{equation*}
  is a homeomorphism.
\end{pro}
\begin{proof}
  \begin{enumerate}
  \item We first prove injectivity. Let $\mathfrak{P}\neq\mathfrak{Q}\in\spc(\T)$, say~$a\in\mathfrak{P}\backslash\mathfrak{Q}$. There exists $i\in I$ and $a_{i}\in \T_{i}$ such that $f_{i}(a_{i})\cong a$ since $f$ is surjective on objects. But then $a_{i}\in f_{i}^{-1}(\mathfrak{P})\backslash f_{i}^{-1}(\mathfrak{Q})$ which implies~$\varphi(\mathfrak{P})\neq\varphi(\mathfrak{Q})$.
  \item For surjectivity, let~$(\mathfrak{P}_{i})_{i}\in\varprojlim\spc(\T_{i})$. Define
    \begin{equation*}
      \mathfrak{P}=\{a\in\T\mid \exists i\in I, a_{i}\in\mathfrak{P}_{i}:a\cong f_{i}(a_{i})\}\subset \T.
    \end{equation*}
    We claim that $\mathfrak{P}$ can also be described as
    \begin{equation*}
      \mathfrak{P}'=\{a\in\T\mid \forall i\in I, a_{i}\in\T_{i}:a\cong f_{i}(a_{i})\Rightarrow a_{i}\in\mathfrak{P}_{i}\}.
    \end{equation*}
    Indeed, if $a\in\mathfrak{P}'$ choose $i\in I$ and $a_{i}\in\T_{i}$ such that $a\cong f_{i}(a_{i})$ which is possible since $f$ is surjective on objects. By definition of $\mathfrak{P}'$ we must have $a_{i}\in\mathfrak{P}_{i}$, and therefore~$a\in\mathfrak{P}$. Conversely, if $a\in\mathfrak{P}$, say $a\cong f_{i}(a_{i})$ with $a_{i}\in\mathfrak{P}_{i}$, and we are given $a_{j}'\in\T_{j}$ such that $a\cong f_{j}(a_{j}')$, let $k\in I$ and $u_{i}:i\to k$,~$u_{j}:j\to k$. We have $f_{k}\T_{u_{i}}(a_{i})\cong f_{i}(a_{i})\cong a\cong f_{j}(a_{j}')\cong f_{k}\T_{u_{j}}(a_{j}')$ and so by assumption on $f$ there exists $u:k\to l$ such that $\T_{uu_{i}}(a_{i})\cong \T_{u}\T_{u_{i}}(a_{i})\cong \T_{u}\T_{u_{j}}(a_{j}')\cong \T_{uu_{j}}(a_{j}')$. The former lies in $\mathfrak{P}_{l}$ hence so does the latter, and this implies~$a_{j}'\in\mathfrak{P}_{j}$.

It is now straightforward to prove that $\mathfrak{P}$ is a prime ideal. For example, let $D:a\to b\to c\to^{+}$ be a triangle in $\T$ with~$a,b\in\mathfrak{P}$. By assumption there exists $i\in I$ and a triangle $D_{i}:a_{i}\to b_{i}\to c_{i}\to^{+}$ in $\T_{i}$ such that~$f_{i}(D_{i})\cong D$. By what we just proved we must then have $a_{i},b_{i}\in\mathfrak{P}_{i}$ and hence also~$c_{i}\in\mathfrak{P}_{i}$. But then~$c\cong f_{i}(c_{i})\in\mathfrak{P}$. Since $\mathfrak{P}$ is clearly closed under translations, this shows that it is a triangulated subcategory.

For thickness we proceed similarly. Let $a,b\in\T$ such that~$a\oplus b\in\mathfrak{P}$. We may find $i\in I$ and $a_{i},b_{i}\in\T_{i}$ such that $a\cong f_{i}(a_{i})$,~$b\cong f_{i}(b_{i})$. Then $f_{i}(a_{i}\oplus b_{i})\cong a\oplus b\in\mathfrak{P}$ thus $a_{i}\oplus b_{i}\in\mathfrak{P}_{i}$ and this implies $a_{i}\in\mathfrak{P}_{i}$ or $b_{i}\in\mathfrak{P}_{i}$, \ie $a\in\mathfrak{P}$ or~$b\in\mathfrak{P}$. Primality is proven in exactly the same way as thickness.

Let $\pi_{i}:\varprojlim\spc(\T_{i})\to \spc(\T_{i})$ be the canonical projection so that~$\pi_{i}\varphi=f_{i}^{-1}$. Then
\begin{align*}
  \pi_{i}\varphi(\mathfrak{P})=f_{i}^{-1}(\mathfrak{P})=f_{i}^{-1}(\mathfrak{P}')=\mathfrak{P}_{i}
\end{align*}
and this completes the proof of surjectivity.
  \item Since $\varphi$ is continuous, it remains to show that it is open. A basis for the topology of $\spc(\T)$ is given by $U(a)=\spc(\T)\backslash\supp(a)$, where $a$ runs through the objects of~$\T$. Fix $a\in\T$, say $a\cong f_{i}(a_{i})$ with some~$a_{i}\in\T_{i}$. We claim that $\varphi(U(a))=\pi_{i}^{-1}(U(a_{i}))$ (which is open hence this would complete the proof).
   
Let $\mathfrak{P}\in U(a)$, which means $f_{i}(a_{i})\cong a\in\mathfrak{P}$, or equivalently, $a_{i}\in f_{i}^{-1}(\mathfrak{P})=\pi_{i}\varphi(\mathfrak{P})$, \ie $\varphi(\mathfrak{P})\in\pi_{i}^{-1}(U(a_{i}))$. Conversely, suppose $(\mathfrak{P}_{i})_{i}\in\pi_{i}^{-1}(U(a_{i}))$, \ie~$a_{i}\in\mathfrak{P}_{i}$. By the proof of surjectivity in part (2),  $(\mathfrak{P}_{i})_{i}=\varphi(\mathfrak{P})$ with $a\in\mathfrak{P}$, \ie~$(\mathfrak{P}_{i})_{i}\in\varphi(U(a))$.
  \end{enumerate}
\end{proof}

\begin{rmk}
  Certainly, these are not the only reasonable conditions on $f$ which allow to deduce a homeomorphism on spectra. For example, it is likely that surjectivity on morphisms could be replaced by a nilfaithfulness assumption inspired by~\cite{balmer:surjectivity}. We mainly chose these conditions with easy applicability in mind.
\end{rmk}

We may apply this result to filtered modules, thereby concluding the second proof of \cref{main-thm}.
\begin{cor}\label{noetherian-reduction}
  If $\rho_{R}^{\bullet}:\spc(\fper(R))\to\spch(R[\beta])$ is a homeomorphism for noetherian rings then it is a homeomorphism for all rings.
\end{cor}
\begin{proof}
  Let $R$ be an arbitrary ring and write it as the filtered colimit of its finitely generated subrings~$R=\varinjlim_{i}R_{i}$. An inclusion $R_{i}\subset R_{j}$ induces a basechange tt-functor $\otimes_{R_{i}}R_{j}:\fper(R_{i})\to\fper(R_{j})$ and we obtain a pseudo-functor $\fper(R_{\bullet}):I\to\ttCatp$ together with a pseudo-natural transformation~$f=\otimes R:\fper(R_{\bullet})\to\fper(R)$. Let us check that $f$ satisfies the assumptions of \cref{spc-continuity-variant}.

  Note first that every free $R$-module comes from a free $R_{i}$-module by basechange, for any~$i$. Also, a morphism between finitely generated free $R$-modules is described by a matrix with entries in~$R$. Adding these finitely many entries to $R_{i}$ we see that morphisms also come from some~$R_{i}$. In particular, this is true for idempotent endomorphisms of finitely generated free $R$-modules. We deduce that finitely generated projective $R$-modules also arise by basechange from some~$R_{i}$. The same is then true for objects and morphisms in $\fmd(R)$ and therefore also in $\fper(R)=\K{\fmd(R)}^{b}$ (\cref{fper-homotopy-derived}). In other words, $f$ is surjective on morphisms. Moreover, a perfect filtered complex is 0 in $\fper(R)$ if and only if it is nullhomotopic and such a homotopy again comes from some $R_{i}$. We conclude that $f$ detects isomorphisms as well.

We may therefore apply \cref{spc-continuity-variant} to deduce a commutative square
\begin{equation*}
  \xymatrix{\spc(\fper(R))\ar[r]\ar[d]_{\rho_{R}^{\bullet}}&\varprojlim_{i}\spc(\fper(R_{i}))\ar[d]^{(\rho_{R_{i}}^{\bullet})_{i}}\\
    \spch(R[\beta])\ar[r]&\varprojlim_{i}\spch(R_{i}[\beta])}
\end{equation*}
where the top horizontal map is a homeomorphism. Since the $R_{i}$ are all noetherian rings, the right vertical map is a homeomorphism by assumption. And the bottom horizontal map is clearly a homeomorphism. We conclude that the left vertical map is too.
\end{proof}

\bibliographystyle{hsiam}
\bibliography{tt-fmod}

\end{document}